\documentclass[a4paper,12pt]{article} 

\title{\'Etale cohomological stability of the \\ moduli space of stable elliptic surfaces}
\date{}
\author{Oishee Banerjee, Jun--Yong Park and Johannes Schmitt}
\usepackage[charter]{mathdesign}
\usepackage{amsmath,graphicx,amsthm,amsfonts,verbatim,mathtools,thmtools,fullpage,tikz-cd,url}
\usepackage[a4paper,margin=3.3cm, headsep=15pt, headheight=2cm]{geometry}
\usepackage[normalem]{ulem}
\usepackage[shortlabels]{enumitem}
\usepackage[all,cmtip]{xy}
\usepackage{hyperref}
\usepackage{subfiles} 
\usepackage{adjustbox,lipsum}
\newtheorem{thm}{Theorem}[section]
\newtheorem{Mthm}[thm]{Main Theorem}
\newtheorem{lem}[thm]{Lemma}
\newtheorem{cor}[thm]{Corollary}
\newtheorem{prop}[thm]{Proposition}
\newtheorem{claim}[thm]{Claim}
\theoremstyle{definition}
\newtheorem{defn}[thm]{Definition}

\newtheorem{conj}[thm]{Conjecture}

\newtheorem{rmk}[thm]{Remark}

\newtheorem{exmp}[thm]{Example}

\newcommand{\iso}{\cong}
\newcommand{\Pic}{\mathrm{Pic}}

\newcommand{\Mg}{\overline{\mathcal{M}}}

\newcommand{\Me}{\overline{\mathcal{M}}_{1,1}}

\newcommand{\Ic}{\mathcal{I}}

\newcommand{\RN}[1]{%
  \textup{\uppercase\expandafter{\romannumeral#1}}%
}

\newcommand{\Pcv}{\mathcal{P}(\vec{\lambda})}

\newcommand{\Pov}{\mathcal{P}(\vec{\Lambda})}

\newcommand{\Ac}{\mathcal{A}}

\newcommand{\Cc}{\mathcal{C}}

\newcommand{\Ec}{\mathcal{E}}
\newcommand{\Hc}{\mathcal{H}}
\newcommand{\Oc}{\mathcal{O}}

\newcommand{\Pc}{\mathcal{P}}
\newcommand{\Lc}{\mathcal{L}}
\newcommand{\Nc}{\mathcal{N}}
\newcommand{\M}{\mathcal{M}}
\newcommand{\Zc}{\mathcal{Z}}

\newcommand{\lambdavec}{{\vec{\lambda}}}
\newcommand{\Lambdavec}{{\vec{\Lambda}}}

\newcommand{\Xc}{\mathcal{X}}
\newcommand{\Xf}{\mathfrak{X}}

\newcommand{\Tc}{\mathcal{T}}

\newcommand{\Rr}{\mathrm{R}}

\newcommand{\Qlb}{\overline{\Qb}_\ell}

\newcommand{\Ql}{\Qb_{\ell}}

\newcommand{\Z}{\mathbb{Z}}
\newcommand{\Q}{\mathbb{Q}}

\newcommand{\N}{\mathbb{N}}

\newcommand{\Pb}{\mathbb{P}}
\newcommand{\Ab}{\mathbb{A}}
\newcommand{\A}{\mathbb{A}}
\newcommand{\Cb}{\mathbb{C}}

\newcommand{\Fb}{\mathbb{F}}
\newcommand{\Gb}{\mathbb{G}}

\newcommand{\Nb}{\mathbb{N}}
\newcommand{\Qb}{\mathbb{Q}}

\newcommand{\Zb}{\mathbb{Z}}
\newcommand{\Fqbar}{\overline{\mathbb{F}}_q}
\newcommand{\et}{{\acute{et}}}

\newcommand{\Spec}{\mathrm{Spec}}

\DeclareMathOperator{\Sym}{Sym}

\DeclareMathOperator{\Frob}{Frob}

\DeclareMathOperator{\Hom}{Hom}

\DeclareMathOperator{\Gal}{Gal}

\begin{document}

    \maketitle


    \begin{abstract}
    We compute the (stable) \'etale cohomology of $\mathrm{Hom}_{n}(C, \mathcal{P}(\vec{\lambda}))$, the moduli stack of degree $n$ morphisms from a smooth projective curve $C$ to the weighted projective stack $\mathcal{P}(\vec{\lambda})$, the latter being a stacky quotient defined by $\mathcal{P}(\vec{\lambda}) := \left[\mathbb{A}^N-\{0\}/\mathbb{G}_m\right]$, where $\mathbb{G}_m$ acts by weights $\vec{\lambda} = (\lambda_0, \cdots, \lambda_N) \in \mathbb{Z}^N_{+}$. Our key ingredient is formulating and proving the \'etale cohomological descent over the category $\Delta S$, the symmetric (semi)simplicial category. An immediate arithmetic consequence is the resolution of the geometric Batyrev--Manin type conjecture for weighted projective stacks over global function fields. Along the way, we also analyze the intersection theory on weighted projectivizations of vector bundles on smooth Deligne--Mumford stacks.
    \end{abstract}
    

    

    \section{Introduction}
    \label{sec:intro}
    
    Fix a base field $K$ and let $C/K$ be a smooth, projective and geometrically connected curve of genus $g$. The moduli space of morphisms of degree $n$ from $C$ to the projective space $\Pb^N$, sometimes dubbed as a Hom-space and denoted by $\Hom_n(C,\Pb^N)$, has been studied extensively for decades, using techniques ranging from scanning maps in the case of $K=\mathbb{C}$ with the Euclidean topology (see e.g. \cite{Segal, CCMM, KS}), to more algebraic approaches in the setting of more general base fields (see e.g. \cite{FW16, Banerjee} and the references therein). 

    In this paper, we consider a generalization: what if we replace the target space by \textit{a weighted projective stack}?

    To elaborate, given a vector $\vec{\lambda} = (\lambda_0, \dotsc, \lambda_N)$ of positive weights $\lambda_i \in \mathbb{Z}_{+}$, we define the $N$-dimensional weighted projective stack $$\Pcv \coloneqq [(\Ab_{x_0, \dotsc, x_N}^{N+1}\setminus 0)/\Gb_m]$$ where $\zeta \in \Gb_m$ acts by $\zeta \cdot (x_0, \dotsc, x_N)=(\zeta^{\lambda_0} x_0, \dotsc, \zeta^{\lambda_N} x_N)$. Now consider the Hom-stack of degree $n \in \mathbb{Z}_{+}$ morphisms from $C$ to $\Pc(\lambdavec)$, which is defined as:

    \vspace{-.2in}

        \begin{align*} \label{eqn:Homnfirstdef}
        \Hom_n(C,\Pcv) \coloneqq \left\{f\colon  C \to \Pcv : f^* \mathcal{O}_{\Pcv}(1) \in \Pic^n C \right\} \\
        = \Big\{ \big(L, [s_0:\ldots: s_N]\big):  L \in \Pic^n C, \,\,\,\, s_i\in H^0(C,L^{\otimes \lambda_i}), \\s_0,\ldots, s_N \text{ have no common zeroes} \Big\}/ \mathbb{G}_m
        \end{align*}
        where $\mathbb{G}_m$ acts on the $i^{th}$ component $H^0(C,L^{\otimes \lambda_i})$ by weight $\lambda_i$. The Hom-stack $\Hom_n(C,\Pcv)$ is a Deligne--Mumford stack by \cite[Theorem 1.1]{Olsson}. 

    The case of $\Pc(4,6)$ is of special interest\footnote{~See Example~\ref{exmp:many_mod} and Section~\ref{sec:Gen_Ell} for other instances of weighted projective stacks of interests.}: noting that the stack $\overline{\mathcal{M}}_{1,1}$ of stable genus $1$ curves is isomorphic to $\Pc(4,6)$, one has that $\Hom_n(C, \Pc(4,6))$ is a moduli stack of \emph{stable elliptic fibrations} over $C$ (see Corollary \ref{Ell_Count}).

    Our goal in this paper is to study the $\ell$-adic cohomology of the space $\Hom_n(C,\Pcv)$. To avoid problems with non-tame stabilizers, we henceforth assume $\mathrm{char}\,\,\, K$ is coprime to $\mathit{l.c.m}(\lambdavec)$ and that $\ell$ is a fixed prime that does not equal $\mathit{char}\,\,\, K$. 

    \vspace{-2ex}
    
    \paragraph{Notations:} A bit of notation before we state our theorems. For a Deligne-Mumford stack $\mathcal{X}$ over $K$, we denote the $\ell$-adic cohomology group with rational coefficients by $H^i(\Xc;\mathbb{Q}_{\ell})$ (warning: this is not the \'etale cohomology with $\Ql$ coefficients, see e.g. \cite[Warning 3.2.1.9]{GL}, or any text on \'etale cohomology of schemes e.g. \cite{Milne}); by the same token a sheaf of $\Ql$ vector spaces is a $\mathbb{Z}_{\ell}$-sheaf $\mathcal{F}= (\mathcal{F}_n)$ and $$H^i(\Xc;\mathcal{F}):= \varprojlim H^i(\Xc;\mathcal{F}_n)\otimes_{\mathbb{Z}_{\ell}} \Ql$$ (similar to the notations set up in \cite[Section 19]{Milne}).  Finally, let us denote a vector space spanned by $\{a_1,\ldots, a_k\}$ over $\Ql$ by $\Ql\{a_1,\ldots, a_k\}$ and write $\Gal(\overline{K}, K)$ for the absolute Galois group of $K$.


    \medskip

   \begin{Mthm}[Cohomological stability]\label{C-cohomThm} Let $\lambda_0, \cdots, \lambda_N\in \mathbb{Z}_{+}$, and let $\Pc(\lambdavec)$ be a weighted projective stack with weights $\lambdavec=(\lambda_0,\cdots, \lambda_N)$.
   Let $C$ be a smooth projective curve of genus $g$. Let $N$ and $n$ be fixed positive integers such that $n\geq 2g$. Set $n_0:=n-2g$. Then there exists a second quadrant spectral sequence, which converges to $H^*(\mathrm{Hom}_n (C,\mathcal{P}(\lambdavec)); \Ql)$ as an algebra, which has the following description. The $E_2$ term is a bigraded algebra that collapses on $E_2^{-p,q}\Big\vert_{p\leq n_0}$. Furthermore, $E_2^{-p,q}\Big\vert_{p\leq n_0}$ is a quotient of the graded commutative $\Ql$-algebra
    $$H^*(J(C);\Ql)[h]/h^N \otimes \wedge \Ql\{t\}\otimes \Sym \Ql\{\alpha_1,\ldots, \alpha_{2g}\},$$ 
    where $H^i(J(C);\Ql)$ has degree $(0,i)$, $h$ has degree $(0,2)$, $t$ has degree $(-1,2N+2)$ 
    and $\alpha_i$ has degree $(-1,2N+1)$ for all $i$, modulo elements of degree $(-i,j)$ with $i>n_0$. Furthermore, the eigenvalues of the action of $\Gal(\overline{K},K)$ on $\Ql\{\alpha_1,\ldots, \alpha_{2g}\}$ are pure of weight $2N+1$; $\Ql\{t\} \cong \Ql(-(N+1))$ and $h$ is a generator of $\Ql(-1)$.  
    \end{Mthm}

    \noindent The special case of rational curve $C=\Pb^1$ deserves a mention in its own right:

    \begin{thm}\label{P-cohomThm} Let $n$ be a positive integer. Then

        $$H^*(\mathrm{Hom}_n (\Pb^1,\mathcal{P}(\lambdavec));\Ql)\cong \frac{\Ql[h]}{h^N} \otimes \wedge\Ql\{t\}$$ where $h$ has cohomological degree $2$ and is a generator of $\Ql(-1)$, and $t$  has cohomological degree $2N+1$ and is a generator $\Ql(-(N+1))$. In particular, we have an isomorphism of $Gal (\overline{K}/K)$-representations:

    \begin{gather*}
        H^i(\mathrm{Hom}_n (\Pb^1,\mathcal{P}(\lambdavec));\Qb_{\ell}) = \begin{cases}
            \Qb_{\ell}(-j) & i=2j, 0\leq j\leq N-1\\ \Qb_{\ell}(-(j+1)) & i=2j+1, N\leq j\leq 2N-1\\ 0 & \text{ otherwise.} 
        \end{cases}
    \end{gather*}

    \end{thm}



    The $\ell$-adic \'etale cohomology with Frobenius weights of the Hom-stack naturally gives the following weighted point count of $\Hom_{n}(C,\Pcv)$ over a finite field $\Fb_q$ via the \textit{Grothendieck-Lefschetz trace formula} for Artin stacks (c.f. Theorem~\ref{SBGLTF}).


    \begin{thm}\label{Gen_Count}
    Let $\Hom_{n}(C,\Pcv)$ be the Hom stack of degree $n \geq 2g$ morphisms from a smooth projective genus $g$ curve $C$ to the $N$-dimensional weighted projective stack $\Pcv = \Pc(\lambda_0, \dotsc, \lambda_N)$ with $|\vec{\lambda}|:=\sum\limits_{i=0}^{N} \lambda_i$. Then the weighted point count of $\Hom_{n}(C,\Pcv)$ over $\Fb_q$ with $\mathrm{char}(\Fb_q) \nmid \lambda_i \in \mathbb{Z}_{+}$ for every $i$ is a finite sum given by
        
        \begin{align*}
       \#_q\left(\Hom_{n}(C,\Pcv)\right) &= q^{|\vec{\lambda}|n + N - 2g} + a_{\frac{1}{2}} \cdot q^{|\vec{\lambda}|n + N - 2g - \frac{1}{2}} + a_{1} \cdot q^{|\vec{\lambda}|n + N - 2g - 1} +  \\& \ldots + a_{i} \cdot q^{|\vec{\lambda}|n + N - 2g - i}+ \ldots
        \end{align*}
        
    \noindent where $i\in \mathbb{Z}_{+}{[\frac{1}{2}]}$, the coefficients $a_{i}$ for $i< n-2g$ are independent of $n$, and for $i\geq n-2g$ we have $a_{i} \cdot q^{|\vec{\lambda}|n + N - 2g - i} \ll q^{|\vec{\lambda}|n + N - 2g}$ .
    \end{thm}

    In the case of rational curve $C=\Pb^1$, the above exact \'etale cohomology gives us the exact weighted point count of $\Hom_{n}(\Pb^1,\Pcv)$ over a finite field $\Fb_q$.  

    \begin{thm}\label{Exac_Count}
    Let $\Hom_{n}(\Pb^1,\Pcv)$ be the Hom stack of degree $n \geq 1$ morphisms from a smooth projective line $\Pb^1$ to the $N$-dimensional weighted projective stack $\Pcv = \Pc(\lambda_0, \dotsc, \lambda_N)$ with $|\vec{\lambda}|:=\sum\limits_{i=0}^{N} \lambda_i$. Then the weighted point count of $\Hom_{n}(\Pb^1,\Pcv)$ over $\Fb_q$ with $\mathrm{char}(\Fb_q) \nmid \lambda_i \in \mathbb{Z}_{+}$ for every $i$ is equal to
        
        \[
        \begin{array}{ll}
        \#_q\left(\Hom_n(\Pb^1,\Pcv)\right) &= \left(\sum\limits_{i=0}^{N} q^{i}\right) \cdot  \left(q^{|\vec{\lambda}|n}-q^{|\vec{\lambda}|n - N}\right) \\
        &\\
        &=q^{|\vec{\lambda}|n - N} \cdot \left(q^{2N} + \dotsb + q^{N+1} - q^{N-1} - \dotsb - 1\right)
        \end{array}
        \]
    \end{thm}

    This result, in particular, effectively answers the geometric Batyrev--Manin type conjecture (the enumeration of rational points of bounded height on varieties over global fields, see \cite{BM}) on weighted projective stacks over global function fields.



    \medskip

    As mentioned earlier, $\Me$ is of special interest; the Deligne--Mumford stack $\Me$ of stable elliptic curves is isomorphic to $\Pc(4,6)$ over $\Spec(\Zb[1/ 6])$ (c.f. Example~\ref{exmp:many_mod}). Consequently, the Hom stack $\Hom_{n}(C,\Pc(4,6))$ over $\text{char} (K) \neq 2,3$ is isomorphic to the moduli stack $\Lc_{12n,g}$ of stable elliptic fibrations with a section and $12n$ nodal singular fibers over the parameterized smooth projective base curve $C_{K}$ of genus $g$ (for further details on the formulation of the moduli stack as Hom-stack we refer to \cite[\S 3]{HP} and \cite[\S 1]{JJ}). 

    \begin{cor}\label{Ell_Count}
    Let $\Lc_{12n,g}$ be the moduli stack of stable elliptic fibrations with a section and discriminant degree $12n$ over the parameterized smooth projective basecurve $C_{\Fb_q}$ of genus $g$. If $\mathrm{char}(\Fb_q) \neq 2,3$ and $n\geq 2g$, then the weighted point count of $\Lc_{12n,g}$ over $\Fb_q$ is the finite sum

        \[
        \begin{array}{ll}
        \#_q\left(\Lc_{12n,g}\right) &= q^{10n + 1 - 2g} + a_{\frac{1}{2}} \cdot q^{10n + 1 - 2g - \frac{1}{2}} + a_{1} \cdot q^{10n + 1 - 2g - 1} + \ldots + a_{i} \cdot q^{10n + 1 - 2g - i} +\cdots
        \end{array}
        \]
    where $i\in \mathbb{Z}_{+}{[\frac{1}{2}]}$, the coefficients $a_{i}$ for $i< n-2g$ are independent of $n$, and for $i\geq n-2g$ we have $a_{i} \cdot q^{10n + 1 - 2g - i} \ll q^{10n + 1 - 2g}$ .
    \end{cor}
    
    \begin{rmk}
    Recall that the weighted point count of $\Lc_{12n,g}$ gives the same number as that of the moduli of semistable elliptic surfaces (c.f. \cite[Proposition 11]{HP}).
    Consequently, we have an estimate of the counting function 
    $$\Nc(\Fb_q(C), 0 < q^{12n} \le B),$$ 
    which counts the number of semistable (i.e., strictly multiplicative reductions) elliptic curves over the parameterized smooth projective genus $g$ basecurve $C_{\Fb_q}$ ordered by height of discriminant $0< ht(\Delta) = q^{12n} \le B$ for $n\geq 2g$ as : 
        \[ \Nc(\Fb_q(C),0 < q^{12n} \le B) = \sum \limits_{n=1}^{\left \lfloor \frac{log_q B}{12} \right \rfloor} 2 \cdot \#_q\left(\Lc_{12n,g}\right) = 2 \cdot \frac{(q^{11-2g} - q^{9-2g})}{(q^{10}-1)} \cdot B^{\frac{5}{6}} + o(B^{\frac{5}{6}}) \]
    where the factor of 2 comes from the hyperelliptic involution.
    \end{rmk}

    This result, in particular, effectively answers the geometric Shafarevich's conjecture (the enumeration of families of algebraic curves (or abelian varieties) with bounded bad reductions over global fields, see \cite{Shafarevich}) on counting semistable elliptic curves over global function fields $\Fb_q(C)$ with $\mathrm{char}(\Fb_q) \neq 2,3$ ordered by bounded height of discriminant.


    \subsection*{Methods}

        The central technique in our cohomology computation follows the method behind the main theorem of \cite{Banerjee}. We construct a suitable \emph{$\Delta S$ object} in the category of Deligne-Mumford stacks, which is a simplicial object enjoying additional properties, whose homotopy colimit is the discriminant locus  \begin{align*} \label{eqn:Homnfirstdef}
        \Zc \coloneqq \left\{f\colon  C \to \Pcv : f^* \mathcal{O}_{\Pcv}(1) \in \Pic^n C \right\} \\
        = \Big\{ \big(L, [s_0:\ldots: s_N]\big):  L \in \Pic^n C, \,\,\,\, s_i\in H^0(C,L^{\otimes \lambda_i}), \\s_0,\ldots, s_N \text{ have at least one common zero} \Big\}/ \mathbb{G}_m
        \end{align*}
        
        This category $\Delta S$, called the \emph{symmetric simplicial category}, is a small category defined by Fiederowicz and Loday in \cite{FL}. It contains $\Delta$ as a subcategory, its objects are that of $\Delta$ and it enjoys much of the key properties of $\Delta$ (being equipped with a natural concept of face and degeneracy maps, all compatible with those in $\Delta$, and sometimes a better substitute for $\Delta$ for simplicial techniques in topology - evidence at hand is Theorem \ref{C-cohomThm}). Among other things, category $\Delta S$ naturally diminishes all the combinatorial complexities that come with $\Delta$ owing to extra automorphisms of its objects; it gives a natural Koszul resolution in the category of constructible $\ell$-adic sheaves on $\Zc$ and that in turn computes the desired cohomology.
        
        It should be highlighted here that whereas cohomological descent for schemes has been well studied in the pioneering work of Deligne (\cite{Deligne}), the translation of it to the world of algebraic stacks is technically much more involved (see \cite{LO} and the references therein); in fact, it is only within the $\infty$-categorical framework that a satisfactory notion of cohomological descent was formulated and proved in the illuminating work of Liu and Zheng (\cite{LZ}). For our purposes, we do not need the full power of \cite{LZ}; we can sidestep the stacky difficulties by proving a statement that exploits representability in algebraic spaces and the concept of proper descent on them (see Lemma \ref{CohomDescentLemmaGen} for a precise statement and its proof).

    \paragraph{Context and connection to other works.}
    
    \begin{itemize}
        \item \textit{Arithmetic statistics:} In the grand scheme of things, this paper is along an interesting connection between the Batyrev--Manin and the Shafarevich conjectures in some very special cases. That is, understanding the arithmetic of rational points on moduli stacks of curves (or abelian varieties) over global fields has direct implications to the enumerations of fibrations of curves (or abelian varieties) over global fields. In a similar regard, the inspiring recent work of \cite{ESZB} initiated the program of understanding the connection between the Batyrev--Manin's conjecture to the \cite{Malle}'s conjecture (the enumeration of number fields of bounded discriminant). Over global function fields, counting semistable elliptic surfaces over $\Pb^1_{\Fb_q}$ has been first addressed in \cite{dJ}; he worked directly with the generalized Weierstrass equations, which works even in characteristic $2$ and $3$ (unlike our method). Works like \cite{HP} and \cite{PS} have similar enumerations over $\text{char} (K) \neq 2,3$ via using motives in $K_0(\mathrm{Stck}_{K})$, the \emph{Grothendieck ring of $K$-stacks} introduced by \cite{Ekedahl} in 2009. 
    
        The strength in our paper lies in methods that entirely deviating from the prior literature. The attractive feature of our enumeration result is its distinctive homotopy theoretic flavour (that techniques of homological stability can be used to resolve geometric Manin's conjecture has been proposed by Ellenberg-Venkatesh, see \cite{EV}). More precisely, taking colimits over the \emph{symmetric simplicial category} $\Delta S$ (which contains $\Delta$ as a subcategory, see \cite{Banerjee}) one gets considerable control over technical difficulties like the class groups being nontrivial in the case of global function fields of higher genera. Furthermore, the central theorem of our paper - \emph{\'etale cohomological stability}, is intrinsically stronger than point count asymptotics.
        
        On a related note, the number of discriminants $\leq B$ of an elliptic curve over $\Zb$ with smooth generic fiber is estimated to be asymptotic to $B^{\frac{5}{6}}$ by \cite{BMc}. The lower order term of order $B^{(7-\frac{5}{27}+\epsilon)/12}$ for counting the stable elliptic curves over $\Qb$ by the bounded height of squarefree discriminants was suggested by the work of \cite{Baier}, improving upon their previous error term in \cite{BB}. 


        \item \textit{\'Etale cohomological stability:} Starting with Quillen's seminal work (see \cite{Quillen}), homological stability has been central to the study of topology of families of spaces/ groups, with vast applications. However, the exploration of \'etale cohomological stability is relatively new and sparse in the literature. Some works that address this are, for example, Farb-Wolfson's work on \'etale hommological stability of configuration spaces of smooth varieties (see \cite{FW18}), Ellenberg-Venkatesh-Westerland's influential work on homological stability of Hurwitz schemes having a fixed Galois group that satisfy certain conditions (see \cite{EVW}), Banerjee's work on spaces admitting symmetric semisimplicial filtration (see \cite{Banerjee}) etc. As the branched covers of the $\Pb^{1}$ are the fibrations with $0$-dimensional fibers, the moduli of fibrations $f: X \to C$ on fibered surfaces $X$ over $C=\Pb^{1}$ (or even for $C$ of higher genus) with $1$-dimensional fibers is the next natural case to work on. Additionally, our method fits many cases: on one hand, for example, our techniques can just as well be applied to prove \'etale cohomological stability of the Hom-stack from higher dimensional smooth projective varieties to $\Pc(\lambdavec)$. And in the case of the domain being of dimension $1$, having computed the stable \'etale cohomology of $\Hom_n (C, \Pc(\lambdavec))$ for any weight vector $\lambdavec$ (see Example~\ref{exmp:many_mod}), we can, for example, make estimates on the number of $\Fb_q$-points on the moduli stack of generalized elliptic fibrations with prescribed level structures (analogous to the work of \cite{HS} via global fields analogy) or multiple markings as in Section~\ref{sec:Gen_Ell}. 

    \end{itemize}
    


    \subsection*{Outline of the paper}
\begin{itemize}
    \item In Section \ref{sec:prelims} we review the definition, properties and examples of weighted projective stacks. We also recall basics on arithmetic of algebraic stacks over finite fields.
    \item We present an in depth analysis of the intersection theory in weighted projective bundles in Section \ref{sec:Inter_wpb}. We prove that the weighted projective bundles formula holds for the Chow ring in Theorem \ref{Thm:weightedprojectivebundle2} as well as for the \'etale cohomology in Corollary \ref{cor:weightedprojectivebundle2_coho}.
    \item We prove Theorem \ref{C-cohomThm} in Section \ref{sec:Coho} using techniques from \cite[Theorem 2]{Banerjee}.
    \item Finally, in Section~\ref{sec:Gen_Ell} we show how our methods can be applied similarly to the moduli stacks of generalized elliptic fibrations with prescribed level structures or multiple marked points thereby extending the work of \cite{HP} to the \'etale cohomological framework of this paper with numerous applications.
\end{itemize}



    \subsection*{Further results}

   \textit{Integral Picard group.} For $N=1$ case with $\Pc(a,b)$, the $\ell$-adic rational cohomology type of the moduli stack $\Hom_{n}(\Pb^{1},\Pc(a,b))$ is a 3-sphere $\mathbb{S}^3$ with an odd-dimensional class at $i=3$ which is independent of both $n$ and $(a,b)$. To say something even more precise about the homotopy type of these moduli stacks, we now compute the integral Picard group of the moduli as it gives the \textit{torsion part of the cohomology}. It turns out that it is always cyclic, and generated by the line bundle $\Oc_{\Hom_n(\Pb^1,\Pcv)}(1)$. 
    

    \begin{thm}\label{thm:Pic_Hom}
    If $\mathrm{char}(K)$ does not divide $\lambda_i \in \mathbb{Z}_{\geq 1}$ for every $i$, then the Picard group of the Hom stack $\Hom_n(\Pb^1,\Pcv)$ is a cyclic group generated by the line bundle $\Oc_{\Hom_n(\Pb^1,\Pcv)}(1)$ and is isomorphic to
    \[
    \mathrm{Pic}(\Hom_n(\Pb^1,\Pcv)) \iso 
    \begin{cases}
    \Zb/(n (\lambda_0 + \lambda_1))\Zb &\text{ for }N=1,\\
    \Zb & \text{ for }N>1.
    \end{cases}
    \]
    \end{thm}

    \begin{proof}
    The proof is at the end of \S \ref{sec:Inter_wpb}.
    \end{proof}

    The torsion Picard group $\mathrm{Pic}(\Hom_n(\Pb^1,\Pc(a,b))) \iso \Zb/((a+b)n)\Zb$ does depend on the degree $n$ and the weights $(a,b)$ (contrary to the fixed rational cohomology type) which leads one to conjecture that its fundamental group (for stacks as in \cite{Noohi2}) is $\pi_1(\Hom_n(\Pb^1,\Pc(a,b))) \iso \Zb/((a+b)n)\Zb$ as confirmed by the work of \'Epshtein in \cite{EP} where he showed $\pi_1(\Hom_n(\Pb^1,\Pc(1,1) \iso \Pb^1)) \iso \Zb/2n\Zb$ . From the perspective of homotopy theory, it is natural to recognize that the homotopy type of the moduli stack is a Lens space. Lens spaces are important as they are the only 3-manifolds with non-trivial finite cyclic fundamental group. 

    \bigskip 

  \noindent  \textit{Rational Chow rings of weighted projectivizations of vector bundles over tame Deligne-Mumford stacks.} For this, assume we have a vector bundle $\mathcal{E}/S$ and a splitting
    $$
    \mathcal{E} = \mathcal{E}_0 \oplus \ldots \oplus \mathcal{E}_N\,,
    $$
    for vector bundles $\mathcal{E}_i$. Given a vector $\lambdavec$ of positive integers, we can again form the weighted projective bundle 
    \begin{equation} \label{eqn:PSElambda}
    \mathcal{P}_S(\mathcal{E}, \lambdavec) = [(\mathrm{Tot}(\mathcal{E}) \setminus 0)/ \Gb_m] \to S\,,
    \end{equation}
    where the torus $\Gb_m$ acts on $\mathcal{E}_i$ with weight $\lambda_i$. 
    
    To state the result for the Chow group of $\mathcal{P}_S(\mathcal{E}, \lambdavec)$, we introduce a notion of \emph{twisted Chern classes}. For this recall that the (standard) Chern polynomial is given by
    $$
    c_t(\mathcal{E}) = 1 + t \cdot c_1(\mathcal{E}) + t^2 c_2(\mathcal{E}) + \ldots
    $$
    For a vector bundle $\mathcal{E}$ splitting into $n+1$ summands as above, the Whitney sum formula implies
    $$
    c_t(\mathcal{E}) = \prod_{i=0}^N c_t(\mathcal{E}_i)\,.
    $$
    Given a vector $\vec \eta = (\eta_0, \ldots, \eta_N) \in \mathbb{Z}_{>0}^{N+1}$ of positive integers, we define the $\vec \eta$-twisted Chern polynomial of $\mathcal{E}$ as
    $$
    c^{\vec \eta}_t(\mathcal{E}) = \prod_{i=0}^N c_{\eta_i t}(\mathcal{E}_i)\,.
    $$
    Similarly, an individual Chern class $c^{\vec \eta}_j(\mathcal{E})$ is defined as the coefficient of $t^j$ in this twisted Chern polynomial. Thinking in terms of the Chern roots of the bundle $\mathcal{E}$, the twisted Chern classes correspond to multiplying the Chern roots of the summand $\mathcal{E}_i$ by $\eta_i$.
    
    \begin{thm} \label{Thm:weightedprojectivebundle2}
    Let $S$ be a smooth Deligne-Mumford stack with a vector bundle $\mathcal{E} = \bigoplus_{i=0}^N \mathcal{E}_i$ and let $\lambdavec \in \mathbb{Z}_{\geq 1}^{N+1}$ be a vector of positive integers. Let $L = \mathrm{lcm}(\lambdavec)$ and consider the vector $\vec \eta = (L/\lambda_0, \ldots, L/\lambda_N)$. Then we have
    \begin{equation} \label{eqn:projectivebundleformula}
        A^*(\mathcal{P}_S(\mathcal{E}, \lambdavec), \Ql) = A^*(S, \Ql)[\zeta] / (\zeta^{N+1} + c^{\vec \eta}_1(\mathcal{E}) \zeta^N + \ldots + c^{\vec \eta}_{N+1}(\mathcal{E}) )\,,
    \end{equation}
    where $\zeta = L \cdot c_1(\mathcal{O}_{\mathcal{P}_S(\mathcal{E}, \lambdavec)}(1))$.
    \end{thm}

\bigskip
    \section{Preliminaries}
    \label{sec:prelims}

    \par We first recall the definition of a weighted projective stack $\Pcv$.

    \begin{defn}\label{def:wtproj} 
    Let $\lambdavec = (\lambda_0, \ldots, \lambda_N) \in \mathbb{Z}_{\geq 1}^{N+1}$ be a vector of $N+1$ positive integers. Consider the affine space $U_\lambdavec = \Ab_{x_0, \dotsc, x_N}^{N+1}$ endowed with the action of $\Gb_m$ with weights $\lambdavec$, i.e.~an element $\zeta \in \Gb_m$ acts by
    \begin{equation}
    \zeta \cdot (x_0, \ldots, x_N) = (\zeta^{\lambda_0} x_0, \ldots, \zeta^{\lambda_N} x_N)\,.
    \end{equation}
    The $N$-dimensional weighted projective stack $\Pcv$ is then defined as the quotient stack
    \[
    \Pcv = \left[(U_{\lambdavec} \setminus \{0\})/\Gb_m \right]\,.
    \]
    \end{defn}

    \begin{rmk}\label{rmk:wtprojtame} 
    When we wish to emphasize the field $K$ of definition of $\Pcv$, we use the notation $\Pc_{K}(\vec\lambda)$. All weighted projective stacks are smooth. The stack $\Pcv$ is Deligne--Mumford if and only if all weights $\lambda_i$ are prime to the characteristic; in this case, $\Pcv$ is in fact tame Deligne--Mumford as in \cite{AOV}. Notice that $\Pc(1,p)$ is not Deligne--Mumford in characteristic $p$ since it has a point with automorphism group $\mu_p$ which is not formally unramified. When $\Pcv$ is Deligne--Mumford, it is an orbifold if and only if $\gcd(\lambda_0, \dotsc, \lambda_N)=1$; this is because $\Pcv$  has generic stabilizer $\mu_{\gcd(\lambda_0, \dotsc, \lambda_N)}$.
    \end{rmk}
    
    The natural morphism $U_\lambdavec \to \Pcv$ is the total space of the \emph{tautological line bundle} $\mathcal{O}_{\Pcv}(-1)$ on $\Pcv$. As in the classical case, we denote by $\mathcal{O}_{\Pcv}(1)$ the dual of this line bundle. 

    \medskip

    The fine modular curves are quintessential moduli stacks. Under mild condition on the characteristic of the base field $K$ (i.e., a field $K$ with $\text{char}(K)$ does not divide $\lambda_i \in \N$ for every $i$), the genus 0 modular curves are isomorphic to the weighted projective stacks $\Pc(a,b)$.

    \begin{exmp}\label{exmp:Mbar_11}\label{exmp:many_mod}
        An example that will play an important role throughout this paper is the moduli stack of stable elliptic curves. When $\text{char}(K)\ne 2,3$, we have an explicit isomorphism 
        $$(\Me)_K \cong [ (\mathrm{Spec}~K[a_4,a_6]-(0,0)) / \Gb_m ] = \Pc_K(4,6)$$ 
        given by the short Weierstrass equation $y^2 = x^3 + a_4x + a_6$, where $\zeta \cdot a_i=\zeta^i a_i$ for $\zeta \in \Gb_m$ and $i=4,6$. See, e.g., \cite[Proposition 3.6]{Hassett}. 

        \medskip

        Similarly, one could consider the stack $\Me[\Gamma]$ of generalized elliptic curves with $[\Gamma]$-level structure, introduced in the work of Deligne and Rapoport \cite{DR} (summarized in \cite[\S 2]{Conrad2} and also in \cite[\S 2]{Niles}) (see Proposition~\ref{prop:Moduli_t}).

        \medskip

        Also, one could consider the stack $\Mg_{1,m}(m-1)$ of $m$-marked $(m-1)$-stable curves of arithmetic genus one formulated originally by the works of \cite{Smyth, Smyth2} (see Proposition~\ref{prop:Moduli_s}).

        \medskip

        For higher genus curves, we recall the notion of quasi--admissible covers whereby the general member of $C$ is \textit{not} an admissible cover of $\Pb^1$ and have been studied in depth by \cite[\S 2.4.]{Stankova} as the closest covers to the original families of stable curves. In this regard, \cite{Fedorchuk} introduced the proper Deligne--Mumford stack $\Hc_{2g}[2g-1]$ of quasi-admissible genus $g \ge 2$ curves. For the case of monic odd--degree hyperelliptic curves with a generalized Weierstrass equation $y^2 = x^{2g+1} + a_{4}x^{2g-1} + a_{6}x^{2g-2} + a_{8}x^{2g-3} + \cdots + a_{4g+2}$ we have $$\Hc_{2g}[2g-1] \iso \Pc(4,6,8,\dotsc,4g+2)$$ by \cite[Proposition 4.2(1)]{Fedorchuk} over $\mathrm{char}(K)=0$ and by \cite[Proposition 5.9]{HP2} over $\mathrm{char}(K)> 2g+1$. The $g=2$ case by $y^2 = x^{5} + a_{4}x^{3} + a_{6}x^{2} + a_{8}x + a_{10}$ with $\Hc_{4}[3] \iso \Pc(4,6,8,10)$ is of special interest as all genus 2 curves are hyperelliptic.



    \end{exmp}

    We recall the definition of \textit{weighted $\Fb_q$--point count} of an algebraic stack $\Xc$.

    \begin{defn}\label{def:wtcount}
        The weighted point count of $\Xc$ over $\Fb_q$ is defined as a sum:
        \[
        \#_q(\Xc)\coloneqq\sum_{x \in \Xc(\Fb_q)/\sim}\frac{1}{|\mathrm{Aut}(x)|},
        \]
        where $\Xc(\Fb_q)/\sim$ is the set of $\Fb_q$--isomorphism classes of $\Fb_q$--points of $\Xc$ (i.e. the set of non--weighted points of $\Xc$ over $\Fb_q$), and we take $\frac{1}{|\mathrm{Aut}(x)|}=0$ when $|\mathrm{Aut}(x)|=\infty$.
    \end{defn}
    
    A priori, the weighted point count can be $\infty$, but when $\Xc$ is of finite type, then the stratification of $\Xc$ by schemes as in \cite[Proof of Lemma 3.2.2]{Behrend} implies that $\Xc(\Fb_q)/\sim$ is a finite set, so that $\#_q(\Xc)<\infty$. 

    \medskip

    The weighted point count $\#_q(\Xc)$ of an algebraic stack $\Xc$ over $\Fb_q$ is \textit{algebro-topological} under the framework of the Weil conjectures as it is equal to the alternating sum of trace of geometric Frobenius. In this regard, we recall the \textit{Grothendieck-Lefschetz trace formula} for Artin stacks by \cite{Behrend,Sun}.

    \begin{thm}[Theorem 1.1. of \cite{Sun}] \label{SBGLTF} 
    Let $\Xc$ be an Artin stack of finite type over $\Fb_q$. Let $\Frob_q$ be the geometric Frobenius on $\Xc$. Let $\ell$ be a prime number different from the characteristic of $\Fb_q$, and let $ \iota : \Qlb \overset{\sim}{\to} \Cb$ be an isomorphism of fields. For an integer $i$, let $H^i_{\et, c}(\Xc_{/\Fqbar};\Qlb)$ be the cohomology with compact support of the constant sheaf $\Qlb$ on $\Xc$. Then the infinite sum regarded as a complex series via $\iota$
            \begin{equation}
            \sum_{i \in \Zb}~(-1)^i \cdot tr\big( \Frob^*_q : H^i_{c}(\Xc_{/\Fqbar};\Qlb) \to H^i_{c}(\Xc_{/\Fqbar};\Qlb) \big)
            \end{equation}
    is absolutely convergent to the weighted point count $\#_q(\Xc)$ of $\Xc$ over $\Fb_q$.
    \end{thm}

    Lastly, we recall that the cohomology in torsion-free coefficient coincide for the fine moduli stack and its coarse moduli space. We show this for the \'etale cohomology over base field $\Fqbar$ in $\Ql$-coefficient by following the proof of the \cite[Proposition 7.3.2]{Sun}.

    \begin{lem} \label{FineCoarseCohomology} Let $\Xf$ be a smooth separated tame Deligne--Mumford stack of finite type over $\Fqbar$ and the coarse moduli map $c : \Xf \to X$ giving the coarse moduli space $X$. Then for all $i$, the pullback map

            \begin{equation*}
            c^* : H^i_{\et}({X}_{/\Fqbar};\Ql) \iso H^i_{\et}({\Xf}_{/\Fqbar};\Ql)
            \end{equation*}

                is an isomorphism.
    \end{lem}

    \begin{proof}
    As $\Xf$ is a smooth separated tame Deligne--Mumford stack of finite type over $\Fqbar$, we can cover $X$ by \'etale charts $U$ such that pull-back of $U$ in $\Xf$ is the quotient stack of an algebraic space by a finite group \cite[Theorem 3.2.]{AOV}. The lemma follows from the $\ell$-adic Leray spectral sequence as in \cite[Theorem 1.2.5]{Behrend} once we have shown that the canonical map $\Ql \to \Rr c_* \Ql$ is an isomorphism. It suffices to show the isomorphism \'etale locally on $X$ and hence we assume $\Xf = [V/G]$ for some algebraic space $V$ under the action of finite group $G$ where $\mathrm{char}(\Fqbar)$ does not divide $|G|$. Let $ q : V \to \Xf $ be the canonical morphism. Observe that we have $\Ql \simeq (q_*\Ql)^G$. As both $q$ and $c \circ q$ are finite maps and $\Qb[G]$ for a finite group $G$ is a semisimple $\Qb$-algebra by the Maschke's Theorem, we acquire
    $$ \Rr c_* \Ql \simeq \Rr c_* (q_* \Ql)^G \simeq ((c \circ q)_* \Ql)^G \simeq \Ql$$ 
    \end{proof}

    \bigskip

    \section{Intersection theory of weighted projective bundles}
    \label{sec:Inter_wpb}

    Many people have studied the cohomology of weighted projective spaces and bundles, starting with \cite{Kawasaki}, where the base $S$ is a point. In \cite{Al Amrani} the author computes the integral cohomology of the coarse moduli space of $\mathcal{P}_S(\mathcal{E}, \lambdavec)$ as a $\mathbb{Q}$-vector space, and its multiplicative structure in case that $\mathcal{E}$ splits as a sum of line bundles and $\lambdavec$ satisfies a certain divisibility condition (with similar results in \'etale cohomology). The divisibility condition was later removed in \cite{BFR}. 

    \medskip

    Our proof of the theorem above is independent of the previous work, and proceeds in two steps: first, in the case where $\mathcal{E}$ splits as a sum of line bundles, the weighted projective bundle $\mathcal{P}_S(\mathcal{E}, \lambdavec)$ admits a natural finite cover to a \emph{non-weighted} projective bundle. Via the usual projective bundle formula, this suffices to find its Chow groups. Here we explicitly use that we work with $\mathbb{Q}$-coefficients. In the second step, we use a splitting theorem to reduce to the first case.

    \medskip

    Let $S$ be a smooth Deligne-Mumford stack and let

    \begin{equation}
        \mathcal{E} = \mathcal{L}_0 \oplus \cdots \oplus \mathcal{L}_N / S
    \end{equation}

    \noindent be a vector bundle on $S$ decomposing into a direct sum of $N+1$ line bundles $\mathcal{L}_i$. We denote by $\mathrm{Tot}(\mathcal{E}) \to S$ the total space of this vector bundle, with zero section $0 \subseteq \mathrm{Tot}(\mathcal{E})$. Given a vector $\lambdavec = (\lambda_0, \ldots, \lambda_N)$ of positive integers, there exists an action of $\Gb_m$ on $\mathrm{Tot}(\mathcal{E})$ locally given by $\zeta \cdot (s_0, \ldots, s_N) = (\zeta^{\lambda_0} s_0, \ldots, \zeta^{\lambda_N} s_N)$. In the following, we want to study the intersection theory of the \emph{weighted projective bundle}
    \begin{equation} \label{eqn:PSElambda_2}
        \mathcal{P}_S(\mathcal{E}, \lambdavec) = [(\mathrm{Tot}(\mathcal{E}) \setminus 0)/ \Gb_m] \to S.
    \end{equation}
    This is a (Zariski) locally trivial bundle with fibre $\Pc(\lambdavec)$ over $S$.
    
    To present the answer, let $\ell = \mathrm{lcm}(\lambdavec)$ be the least common multiple of the $\lambda_i$, consider the vector $\vec \eta = (\ell/\lambda_0, \ldots, \ell/\lambda_N)$ and the modified bundle
    \begin{equation}
        \mathcal{E}_\eta = \mathcal{L}_0^{\otimes \eta_0} \oplus \cdots \oplus \mathcal{L}_N^{\otimes \eta_N} / S\,.
    \end{equation}
    Denote by $\mathbb{P}(\mathcal{E}_\eta) \to S$ the (unweighted) projective bundle associated to this vector bundle. Then there exists a natural map
    \begin{equation}
        \Phi : \mathcal{P}_S(\mathcal{E}, \lambdavec) \to \mathbb{P}(\mathcal{E}_\eta), [s_0 : \ldots : s_N] \mapsto [s_0^{\otimes \eta_0} : \ldots : s_N^{\otimes \eta_N}]
    \end{equation}
    over $S$.
    \begin{thm} \label{Thm:weightedprojectivebundle1}
    The map $\Phi$ is proper, flat and quasi-finite of degree $d=\ell^N/ \prod_{i=0}^N \lambda_i$ and the pullback
    \begin{equation*}
        \Phi^* : A^*(\mathbb{P}(\mathcal{E}_\eta), \mathbb{Q}) \to A^*(\mathcal{P}_S(\mathcal{E}, \lambdavec), \mathbb{Q})
    \end{equation*}
    induces an isomorphism on the Chow groups with $\mathbb{Q}$-coefficients, whose inverse is given by $1/d \cdot \Phi_*$. In particular, we have
    \begin{equation} \label{eqn:projectivebundleformula_1}
        A^*(\mathcal{P}_S(\mathcal{E}, \lambdavec), \mathbb{Q}) = A^*(S, \mathbb{Q})[\zeta] / (\zeta^{N+1} + c_1(\mathcal{E}_\eta) \zeta^N + \ldots + c_{N+1}(\mathcal{E}_\eta) )\,,
    \end{equation}
    where $\zeta = \ell \cdot c_1(\mathcal{O}_{\mathcal{P}_S(\mathcal{E}, \lambdavec)}(1))$ and where the Chern classes of $\mathcal{E}_\eta$ can be computed as
    \[
    c_i(\mathcal{E}_\eta) = e_i(\eta_0 c_1(\mathcal{L}_0), \ldots, \eta_N c_1(\mathcal{L}_N))\,,
    \]
    with $e_i$ the $i$-th elementary symmetric polynomial.
    \end{thm}
    For the proof, we need some more preparatory results.
    \begin{lem} \label{Lem:partialquotient}
    Let $G$ be an algebraic group acting on varieties $X,Y$ over $K$ such that the action on $X$ is transitive and let $x_0 \in X$ be a $K$-point with stabilizer $G_{x_0}$. Then there exists a natural isomorphism
    \begin{equation}
        f : [X \times Y / G] \xrightarrow{\sim} [Y / G_{x_0}]
    \end{equation}
    whose inverse is induced by the $G_{x_0}$-equivariant map
    \[
    Y \to \{x_0\} \times Y \subseteq X \times Y\,.
    \]
    \end{lem}
    \begin{proof}
    To construct $f$ consider the incidence variety
    \begin{equation} \label{eqn:finverse}
        I = \{(g, x, y) : g x = x_0\} \subseteq G \times X \times Y
    \end{equation}
    Let $p : I \to X \times Y$ be the projection on the second and third factor. By the assumption on transitivity, the map $p$ is surjective. Moreover, for the action of the group $G_{x_0}$ on $I$ given by left-translation on the factor $G$ (and the \emph{trivial} action on $X,Y$), we claim that $p$ is a principal $G_{x_0}$-bundle. An fppf cover of $X$ trivializing this bundle is given by the projection $p : I \to X \times Y$ itself. 
    
    Indeed, the fibre product $F = I \times_{X \times Y} I$ parameterizes tuples $(g, g', x, y) \in G \times G \times X \times Y$ such that $g x = x_0$, $g' x = x_0$. Setting $\widetilde{g} = g' \circ g^{-1} \in G_{x_0}$ this is equivalent to the data of $(g, \widetilde{g}, x,y)$ such that $g x = x_0$ and such that $\widetilde{g} \in G_{x_0}$, which defines a trivial $G_{x_0}$-bundle over $I$ as desired.
    
    Now we note that the map $\widetilde{f} : I \to Y, (g,x,y) \mapsto g y$ is $G_{x_0}$-equivariant (with respect to $G_0$-action on the target induced by its given $G$-action). Using that $p: I \to X \times Y$ is a principal $G_{x_0}$-bundle together with the definition of the stack $[Y / G_{x_0}]$ it gives rise to a map $\overline{f} : X \times Y \to [Y / G_{x_0}]$. On the other hand, it is not hard to see that this map is invariant under the $G$-action on $X \times Y$ and thus factors through the quotient stack $[X \times Y / G]$ via the desired map $f : [X \times Y / G] \to [Y / G_{x_0}]$. It is then straightforward to check that the map \eqref{eqn:finverse} induces the inverse map of $f$.
    \end{proof}
    \begin{prop} \label{Prop:PlambdavectChowKunneth}
    Given a vector $\lambdavec \in \mathbb{Z}_{>0}^{N+1}$ of positive integers, the weighted projective space $\Pc(\lambdavec)$ has Chow group
    \[
    A^*(\Pc(\lambdavec), \mathbb{Q}) = \mathbb{Q}[\zeta]/(\zeta^{N+1})\,,
    \]
    where $\zeta = c_1(\mathcal{O}_{\Pc(\lambdavec)}(1))$ is the first Chern class of the bundle $\mathcal{O}_{\Pc(\lambdavec)}(1)$. Moreover, the space $\Pc(\lambdavec)$ satisfies the \emph{Chow-Künneth property}, i.e. for any smooth Deligne-Mumford stack $X$ the natural map
    \[
    \underbrace{A^*(X, \mathbb{Q}) \otimes_{\mathbb{Q}} A^*(\Pc(\lambdavec), \mathbb{Q})}_{=A^*(X, \mathbb{Q})[\zeta]/(\zeta^{N+1})} \to A^*(X \times \Pc(\lambdavec), \mathbb{Q})
    \]
    is an isomorphism.
    \end{prop}
    Note that the intersection theory of $\Pc(\lambdavec)$ was studied before in various settings (see \cite{Kawasaki, Al Amrani, BFR}). We give a self-contained proof in the language of Chow groups with $\mathbb{Q}$-coefficients, which gives us the additional Chow-Künneth property that we need later.
    \begin{proof}[Proof of Proposition \ref{Prop:PlambdavectChowKunneth}]
    For our proof we want to use the fact that $\Pc(\lambdavec)$ has a cellular decomposition (similar to the concept in \cite{FultonIntersectionTheory}), which in our context means a locally closed stratification in stacks which are isomorphic to finite quotients of affine spaces $\mathbb{A}^n$. The closures of these strata then form a basis of $A^*(\Pc(\lambdavec), \mathbb{Q})$.
    
    To make this more precise, we first note that for dimension reasons we indeed have $\zeta^{N+1}=0$, so that there is a well-defined map from $\mathbb{Q}[\zeta]/(\zeta^{N+1})$ to $A^*(\Pc(\lambdavec), \mathbb{Q})$. Our goal is to show that this map is surjective and injective.
    
    For the surjectivity, consider the open substack 
    $$
    \mathcal{U}_0 = \{[x_0 : \ldots : x_N] : x_0 \neq 0\} \subseteq \Pc(\lambdavec)
    $$
    with complement $\mathcal{Z}_0 \cong \Pc(\lambda_1, \ldots, \lambda_N)$. Then we have the excision sequence
    \begin{equation} \label{eqn:excision1}
        A_*(\mathcal{Z}_0, \mathbb{Q}) \xrightarrow{i_*} A_*(\Pc(\lambdavec), \mathbb{Q}) \to A_*(\mathcal{U}_0, \mathbb{Q}) \to 0\,.
    \end{equation}
    By induction (starting with the trivial case $N=0$) we have $A_*(\mathcal{Z}_0, \mathbb{Q}) \cong \mathbb{Q}[\zeta]/(\zeta^{N+1})$. Moreover, since $\mathcal{Z}_0 = \{x_0 = 0\}$ is the zero set of the section $x_0$ of $\mathcal{O}_{\Pc(\lambdavec)}(\lambda_0)$, we have that the image of the map $i_*$ equals the image of $(\lambda_0 \zeta) \cdot \mathbb{Q}[\zeta]/(\zeta^{N}) \subseteq \mathbb{Q}[\zeta]/(\zeta^{N+1})$. On the other hand, the open substack $\mathcal{U}_0$ is isomorphic to $[\Gb_m \times \mathbb{A}^N / \Gb_m]$. Then observe that the action of $\Gb_m$ on itself of weight $\lambda_0$ is transitive and the stabilizer of $1$ equals $\mu_{\lambda_0}$. Thus using Lemma \ref{Lem:partialquotient} we have 
    $$
    \mathcal{U}_0 \cong [\mathbb{A}^N / \mu_{\lambda_0}].
    $$
    This is a vector bundle over $B \mu_{\lambda_0}$ of rank $N$ and thus its Chow group is isomorphic to $A^*(B \mu_{\lambda_0}, \mathbb{Q}) = \mathbb{Q} \cdot [B \mu_{\lambda_0}]$ and hence trivial. Thus the excision sequence \eqref{eqn:excision1} has the form
    \[
    \mathbb{Q}[\zeta]/(\zeta^{N}) \xrightarrow{\cdot \lambda_0 \zeta} A_*(\Pc(\lambdavec), \mathbb{Q}) \to \mathbb{Q} \to 0\,,
    \]
    which implies that $\mathbb{Q}[\zeta]/(\zeta^{N+1})$ surjects onto the Chow group of $\Pc(\lambdavec)$. For injectivity, note that the kernel of the natural map $\mathbb{Q}[\zeta] \to A_*(\Pc(\lambdavec), \mathbb{Q})$ must contain the ideal $(\zeta^{N+1})$ and thus be of the form $(\zeta^m)$ for some $m \leq N+1$. However, since $\zeta^N \neq 0$ as it is a zero-cycle of positive degree, the kernel is indeed equal to $(\zeta^{N+1})$.
    
    Given the stratification into $\mathcal{U}_0, \mathcal{Z}_0$ as above, by \cite[Proposition 2.10]{BaeSchmitt} and an inductive argument, the Chow-Künneth property of $\Pc(\lambdavec)$ follows if we can show the analogous property for $\mathcal{U}_0 \cong [\mathbb{A}^N / \mu_{\lambda_0}]$. This amounts to showing that
    $$
    A^*(X, \mathbb{Q}) \cong A^*(X \times [\mathbb{A}^N / \mu_{\lambda_0}], \mathbb{Q}).
    $$
    But again $X \times [\mathbb{A}^N / \mu_{\lambda_0}]$ is a vector bundle over $X \times B \mu_{\lambda_0}$ and thus the Chow groups of the two spaces agree. On the other hand, the coarse moduli space of $X \times B \mu_{\lambda_0}$ agrees with the coarse moduli space $|X|$ of $X$ and since the Chow group with $\mathbb{Q}$-coefficients can be computed on such a coarse moduli space, we get the desired chain of isomorphisms
    $$
    A^*(X \times [\mathbb{A}^N / \mu_{\lambda_0}], \mathbb{Q}) \cong A^*(X \times B \mu_{\lambda_0}, \mathbb{Q}) \cong A^*(|X|, \mathbb{Q}) \cong A^*(X, \mathbb{Q})\,.
    $$
    \end{proof}
    \begin{proof}[Proof of Theorem \ref{Thm:weightedprojectivebundle1}]
    Since both $\Pc(\mathcal{E}, \lambdavec)$ and $\mathbb{P}(\mathcal{E}_\eta)$ are Zariski-locally trivial bundles over $S$, the properties of the map $\Phi$ which are local on the target (proper, flat, quasi-finite and the degree) can be verified for the trivial base $S = \mathrm{Spec}(K)$. Thus consider the map $\Phi : \Pc(\lambdavec) \to \mathbb{P}^N$. The open subsets $U_i = \{x_i \neq 0\} \subseteq \mathbb{P}^N$ form a Zariski cover and the preimage of $U_i$ under $\Phi$ is precisely $\mathcal{U}_i = \{s_i \neq 0\} \subseteq \Pc(\lambdavec)$, which is isomorphic to $[\mathbb{A}^N / \mu_{\lambda_i}]$. Then we have a diagram
    \[
    \begin{tikzcd}
    \mathbb{A}^N \arrow[dr]  \arrow[d] & \\
    \mathcal{U}_i \arrow[r] & U_i \cong \mathbb{A}^N
    \end{tikzcd}
    \]
    and the diagonal map $\mathbb{A}^N \to \mathbb{A}^N$ is the finite, flat map of degree $\ell^N / \prod_{j \neq i} \lambda_j$ given by
    \[
    (s_0, \ldots, \widehat{s_i}, \ldots, s_N) \mapsto (s_0^{\ell/\lambda_0}, \ldots, \widehat{s_i^{\ell/\lambda_i}}, \ldots , s_N^{\ell/\lambda_N})\,,
    \]
    where the hat indicates that we omit the $i$-th entry of the vector.
    Since the map $\mathbb{A}^N \to \mathcal{U}_i$ is of degree $\lambda_i$, this shows all local properties of the map $\Phi$.
    
    Returning to the case of a general base $S$, it remains to show that $\Phi^*$ is an isomorphism with inverse $1/d \cdot \Phi_*$. Indeed, once this is established, the presentation \eqref{eqn:projectivebundleformula_1} is simply the standard formula for the Chow group of the projective bundle $\mathbb{P}(\mathcal{E}_\eta)$, where in addition we use that the natural line bundle $\mathcal{O}_{\mathbb{P}(\mathcal{E}_\eta)}(1)$ pulls back to $\mathcal{O}_{\mathcal{P}_S(\mathcal{E}, \lambdavec)}(\ell)$ under $\Phi$.
    
    To complete the proof, we note that the composition $(1/d \cdot \Phi_*) \circ \Phi^*$ is clearly equal to the identity by the projection formula (see \cite[Example 1.7.4]{FultonIntersectionTheory} for the argument in the case of schemes). This shows that $\Phi^*$ is injective, so we conclude by proving its surjectivity.
    
    We begin by noting, that since we have a diagram
    \[
    \begin{tikzcd}
    \mathcal{P}_S(\mathcal{E}, \lambdavec) \arrow[rr,"\Phi"] \arrow[dr,"\pi",swap] & & \mathbb{P}(\mathcal{E}_\eta) \arrow[dl, "\pi'"]\\
    & S &
    \end{tikzcd}
    \]
    it follows that the image of $\Phi^*$ contains all classes pulled back from $S$ itself. On the other hand, the natural line bundle $\mathcal{O}_{\mathbb{P}(\mathcal{E}_\eta)}(1)$ pulls back to $\mathcal{O}_{\mathcal{P}_S(\mathcal{E}, \lambdavec)}(\ell)$, so that the class
    $$\widetilde{\xi} = c_1(\mathcal{O}_{\mathcal{P}_S(\mathcal{E}, \lambdavec)}(1)) \in A^1(\mathcal{P}_S(\mathcal{E}, \lambdavec), \mathbb{Q})$$
    is likewise contained in the image of $\Phi^*$. We want to show that these two types of classes generate $A^*(\mathcal{P}_S(\mathcal{E}, \lambdavec), \mathbb{Q})$ as a $\mathbb{Q}$-algebra, i.e. that the natural map
    \[
    \Psi_S : A^*(S, \mathbb{Q})[\widetilde{\zeta}] \to A^*(\mathcal{P}_S(\mathcal{E}, \lambdavec), \mathbb{Q}), \quad \alpha \cdot \widetilde{\zeta}^m \mapsto \pi^*(\alpha) \cdot \widetilde{\zeta}^m
    \]
    is surjective.
    For this, let $V \subseteq S$ be an open substack on which all summands $\mathcal{L}_i$ of  $\mathcal{E}$ are trivial. Then we have $\pi^{-1}(V) \cong V \times \Pc(\lambdavec)$ is a trivial product over $V$. Denoting $K = S \setminus V$ the complement of $V$, we have two excision sequences
    \begin{equation}
    \begin{tikzcd}
    A_*(\pi^{-1}(K), \mathbb{Q}) \arrow[r] & A_*(\mathcal{P}_S(\mathcal{E}, \lambdavec), \mathbb{Q}) \arrow[r] & A_*(V \times \Pc(\lambdavec), \mathbb{Q}) \arrow[r] & 0\\
    A_*(K, \mathbb{Q})[\widetilde{\zeta}] \arrow[u, "\Psi_K"] \arrow[r]& A_*(S, \mathbb{Q})[\widetilde{\zeta}] \arrow[r] \arrow[u, "\Psi_S"] & A_*(V, \mathbb{Q})[\widetilde{\zeta}] \arrow[u, "\Psi_V"] \arrow[r] & 0
    \end{tikzcd}
        \,.
    \end{equation}
    Here the bottom row of the diagram is obtained from the excision sequence of $V \subseteq S$ by tensoring with $\mathbb{Q}[\widetilde{\zeta}]$, which is a right-exact operation. Now we see that the map $\Psi_V$ is surjective by Proposition \ref{Prop:PlambdavectChowKunneth}, whereas the surjectivity of $\Psi_K$ follows by Noetherian induction. By the four lemma, it follows that $\Psi_S$ is surjective, which finishes the proof.
    \end{proof}

    \begin{claim}[Stacky Leray-Hirsch]\label{StackyLeray-Hirsch}
        Let $S$ be a smooth Deligne-Mumford stack and let $\Pc_S(\Ec)$ be the weighted projective bundle on $S$ defined as before. Then $H^*(\Pc_S(\Ec);\Ql) \cong H^*(S;\Ql) \otimes H^*(\Pc(\lambdavec);\Ql)$ as $H^*(S;\Ql)$-modules.
    \end{claim}
    \begin{proof}
        We first observe that the Leray-Hirsch condition is obviously satisfied; i.e. the elements $1, \zeta, \cdots, \zeta^N$ of $H^*(\Pc_S(\Ec);\Ql)$ in cohomological degrees $0, 2, 4, \cdots, 2N$ restrict to form a basis of the cohomology of the fibre $H^*(\Pc(\lambdavec);\Ql)$. Therefore in the derived category of $\ell$-adic sheaves over $S$ we have a map induced by the cohomology classes (by the decomposition theorem):
        $$\oplus_{0\leq i\leq N} \underline{\Ql}_S[-2i] \to R\pi_*\underline{\Ql}_{\Pc_S(\Ec)}$$ which is in fact an isomorphism thanks to the Leray-Hirsch condition. Now the claim immediately follows by applying the global section functor to this complex and taking cohomology.
    \end{proof}

    \subsection{Proof of Theorem~\ref{Thm:weightedprojectivebundle2}}
    Based on the theorem above, we can now also compute the Chow group of a weighted projective bundle, where the bundle does not split into line bundles, but rather into sub-vector bundles of possibly higher ranks.
    
    For the proof, we need a variant of the \emph{splitting principle} inspired by an \\ \href{https://mathoverflow.net/q/87724}{answer of Angelo Vistoli on mathoverflow}.
    \begin{lem} \label{Lem:splittingprinciple}
    Given a Deligne-Mumford stack $S$ and a vector bundle $\mathcal{E}$ on $S$, there exists a flat morphism $f: \widehat{S} \to S$ such that for any morphism $W \to S$ the pullback 
    $$f^* : A^*(W, \mathbb{Q}) \to A^*(\widehat{S} \times_S W, \mathbb{Q})$$
    is injective and such that $f^* \mathcal{E}$ splits as a direct sum of line bundles.
    \end{lem}
    \begin{proof}
    Let $P \to S$ be the frame bundle of $\mathcal{E}$, whose fibre over $s \in S$ is the set of all bases of the vector space $\mathcal{E}_s$. It carries a natural action of $\mathrm{GL}_r$. Let $T \subseteq B \subseteq \mathrm{GL}_r$ be the maximal torus and Borel subgroup of $\mathrm{GL}_r$, respectively. Then we claim that the map $f : \widehat{S} = [P / T] \to S$ satisfies the properties of the lemma. Indeed, since $\widehat{S}$ over any point $s \in S$ parameterizes a basis of $\mathcal{E}_s$ up to scaling, the pullback  $\mathcal{E}|_{\widehat{S}}$ naturally splits into line bundles. On the other hand, the map $f$ is the composition $[P/T] \to [P/B] \to S = [P/\mathrm{GL}_r]$ of an affine bundle and a surjective projective morphism, both of which have injective pullbacks in Chow (which remains true after base change with $W \to S$). 
    \end{proof}
    \begin{proof}[Proof of Theorem \ref{Thm:weightedprojectivebundle2}]
    When the rank of all bundles $\mathcal{E}_i$ is $1$, the result is precisely Theorem \ref{Thm:weightedprojectivebundle1}. In the more general setting, we can use the fact that the bundles $\mathcal{E}_i$ split into line bundles Zariski locally and repeat the excision sequence arguments from the proof of Theorem \ref{Thm:weightedprojectivebundle1} to conclude that the Chow group of $\mathcal{P}_S(\mathcal{E}, \lambdavec)$ is generated as a $\mathbb{Q}$-algebra by classes pulled back from $S$ and by the class $\zeta$. In other words, we have a natural surjection 
    \begin{equation} \label{eqn:Psimapsplitting}
    \Psi: A^*(S, \mathbb{Q})[\zeta] \to A^*(\mathcal{P}_S(\mathcal{E}, \lambdavec), \mathbb{Q})\,.
    \end{equation}
    Given $f : \widehat{S} \to S$ as in Lemma \ref{Lem:splittingprinciple}, we can ensure that all $f^* \mathcal{E}_i$ split into line bundles (applying the lemma $N+1$ times if necessary). Then we can compute the Chow group of $\mathcal{P}_{\widehat{S}}(f^* \mathcal{E}, \lambdavec) = \mathcal{P}_S(\mathcal{E}, \lambdavec) \times_S \widehat{S}$ and have a sequence of maps
    \[
    A^*(S, \mathbb{Q})[\zeta] \xrightarrow{\Psi} A^*(\mathcal{P}_S(\mathcal{E}, \lambdavec), \mathbb{Q}) \xrightarrow{f^*} A^*(\mathcal{P}_S(\mathcal{E}, \lambdavec), \mathbb{Q}) = A^*(\widehat{S}, \mathbb{Q})[\zeta] / (\zeta^{N+1} + \ldots + f^* c_{N+1}^{\vec \eta}(\mathcal{E}))
    \]
    with $f^*$ injective. It follows that the element
    \[
    Q = \zeta^{N+1} + \zeta^N c_{1}^{\vec \eta}(\mathcal{E}) \ldots +  c_{N+1}^{\vec \eta}(\mathcal{E})
    \]
    must be in the kernel of $\Psi$. To show that it generates the kernel (as an ideal), note that the quotient $A^*(S, \mathbb{Q})[\zeta] / (Q)$ has a direct sum decomposition
    \begin{equation} \label{eqn:ASmodQ}
    A^*(S, \mathbb{Q})[\zeta] / (Q) = \bigoplus_{i=0}^N A^*(S, \mathbb{Q}) \cdot \zeta^i
    \end{equation}
    into $N+1$ copies of $A^*(S, \mathbb{Q})$. Let $\pi : \mathcal{P}_S(\mathcal{E}, \lambdavec) \to S$ be the projection to the base, then we have a linear map
    \[
    G : A^*(S, \mathbb{Q})[\zeta] / (Q) \to \bigoplus_{i=0}^N A^*(S, \mathbb{Q}), \alpha \mapsto \left( \pi_*(\Psi(\alpha)\cdot \zeta^i) \right)_{i=0, \ldots, N}\,.
    \]
    We claim that with respect to the direct sum decomposition \eqref{eqn:ASmodQ}, the map $G$ is given by multiplication with a matrix of elements in $A^*(S, \mathbb{Q})$ with entries $1$ on the anti-diagonal, and vanishing entries above this anti-diagonal. Assuming the claim, the map $G$ is injective (up to reordering it is a triangular base change) and since it factors via $\Psi : A^*(S, \mathbb{Q})[\zeta] / (Q) \to A^*(\mathcal{P}_S(\mathcal{E}, \lambdavec), \mathbb{Q})$, the map $\Psi$ must also be injective. We have already seen it to be surjective, so it is an isomorphism and thus the theorem is proven.
    
    To show the claim, note that $\pi_* \zeta^{N} = [S]$, whereas $\pi_* \zeta^i = 0$ for $i < N$. It follows that for an element $\alpha = \sum_{j=0}^N \alpha_j \zeta^j$ in the domain of $G$ (with $\alpha_j \in A^*(S, \mathbb{Q})$), we have that the $i$th component of $G(\alpha)$ is given by
    \[
    G(\alpha)_i = \sum_{j=0}^N \pi_*(\alpha_j \cdot \zeta^{j+i}) = \sum_{j=0}^N \alpha_j \cdot \underbrace{\pi_*(\zeta^{j+i})}_{=0\text{ for }j<N-i}\,.
    \]
    But as stated before, we have $\pi_*(\zeta^{j+i}) = [S]$ for $j=N-i$ and $\pi_*(\zeta^{j+i})=0$ for $j<N-i$, which is precisely the shape of the matrix describing $G$ that was claimed.
    \end{proof}

    The weighted projective bundle formula in the rational Chow ring holds in the $\ell$-adic rational \'etale cohomology ring as well.

    \begin{cor} \label{cor:weightedprojectivebundle2_coho}
    Let $S$ be a smooth Deligne-Mumford stack with a vector bundle $\mathcal{E} = \bigoplus_{i=0}^N \mathcal{E}_i$ and let $\lambdavec \in \mathbb{Z}_{\geq 1}^{N+1}$ be a vector of positive integers. Let $L = \mathrm{lcm}(\lambdavec)$ and consider the vector $\vec \eta = (L/\lambda_0, \ldots, L/\lambda_N)$. Then we have
    \begin{equation} \label{eqn:projectivebundleformula_Coho}
        H^*(\mathcal{P}_S(\mathcal{E}, \lambdavec), \Ql) = H^*(S, \Ql)[\zeta] / (\zeta^{N+1} + c^{\vec \eta}_1(\mathcal{E}) \zeta^N + \ldots + c^{\vec \eta}_{N+1}(\mathcal{E}) )\,,
    \end{equation}
    where $\zeta = L \cdot c_1(\mathcal{O}_{\mathcal{P}_S(\mathcal{E}, \lambdavec)}(1))$.
    \end{cor}

    \begin{proof}
    The rational cohomology of the smooth separated tame Deligne--Mumford stack is the same as its coarse space by Lemma \ref{FineCoarseCohomology}. This allows us to use the description of the cohomology ring of a weighted projective bundle as an algebra over the cohomology ring of the base as in \cite[Theorem 6.2]{BFR} (see also \cite[\S III]{Al Amrani}) which are proved via the application of Leray-Hirsch Theorem.
    
    These references finish the case when the bundle $\mathcal{E}$ splits as a sum of line bundles. As before we can conclude the general case by applying the splitting principle. For this, we repeat the arguments of the section above, replacing Chow groups with  cohomology with $\mathbb{Q}_\ell$-coefficients. The global shape of the argument is the same, but we replace various technical sub-claims as follows:
    \begin{itemize}
        \item For the map $f : \widehat{S} \to S$ on which the pullback of $\mathcal{E}$ splits, we take the same construction $f: \widehat{S} = [P/T] \to S$ as in Lemma \ref{Lem:splittingprinciple}. Again, the pullback in cohomology via $f$, which is a composition of an affine bundle and a full flag bundle, is injective (using the Leray-Hirsch theorem).
        \item Instead of using excision theorems, the surjectivity of the map
        \[\Psi: H^*(S, \mathbb{Q}_\ell)[\zeta] \to H^*(\mathcal{P}_S(\mathcal{E}, \lambdavec), \mathbb{Q}_\ell),\]
        defined analogously to \eqref{eqn:Psimapsplitting}, follows from Claim \ref{StackyLeray-Hirsch}. The same claim also proves that $\Psi$ becomes an isomorphism when dividing its domain by the ideal generated by $Q$.
    \end{itemize}
    This last statement finishes the proof.
    \end{proof}

    We conclude this Section by proving Theorem~\ref{thm:Pic_Hom}.

    \subsection{Proof of Theorem~\ref{thm:Pic_Hom}}

    \begin{proof} 
    We have that $\Hom_n(\Pb^1,\Pcv)$ is naturally an open substack of the weighted projective stack $\Pc(\Lambdavec)$ of dimension $|\vec{\lambda}|n+N$, whose complement $Z \subseteq \Pc(\Lambdavec)$ is the locus of points $[s_0: \ldots s_N]$ where all $s_i$ \emph{do} vanish simultaneously at some $q \in \Pb^1$. We have an isomorphism $\mathrm{Pic}(\Pc(\Lambdavec)) = \mathbb{Z}$ by \cite[Proposition 6.4.]{Noohi} and for the Picard group of $\Hom_n(\Pb^1,\Pcv)$, we claim that the codimension of $Z$ in $\Pc(\Lambdavec)$ is precisely $N$, implying that for $N \geq 2$ we have
    \[\mathrm{Pic}(\Hom_n(\Pb^1,\Pcv)) = \mathrm{Pic}(\Pc(\Lambdavec) \setminus Z) = \mathrm{Pic}(\Pc(\Lambdavec)) = \mathbb{Z}\]
    by \cite[(ii) Theorem 2.1.4.]{Fringuelli}. To bound the codimension of $Z$, consider the incidence variety
    \[\widehat Z = \{(q,[s_0, \ldots, s_N]) \in \Pb^1 \times \Pc(\Lambdavec) : s_i(q) = 0 \text{ for }i=0, \ldots, N\} \subseteq \Pb^1 \times \Pc(\Lambdavec).\]
    For a fixed $q \in \Pb^1$, the set of points $(s_0, \ldots, s_N) \in \Lambdavec$ such that all $s_i$ vanish at $q$ is a linear subspace of codimension $N+1$ in $\Lambdavec$, and in fact, the projection $\widehat Z \to \Pb^1$ is a projective bundle of the corresponding rank. In particular, we have that $\widehat Z \subset \Pb^1 \times \Pc(\Lambdavec)$ is irreducible of codimension $N+1$. On the other hand, we have that $Z \subset \Pc(\Lambdavec)$ is the image of $\widehat Z$ under the projection $\Pb^1 \times \Pc(\Lambdavec) \to \Pc(\Lambdavec)$. Since the fibres of $\widehat Z \to Z$ are finite, it follows that the codimension of $Z$ is indeed $N$. 

    On the one hand, for $N=1$ with $\Pc(\lambda_0,\lambda_1)$ we have that the subvariety $Z$ of $\Pc(\Lambdavec)$ is cut out by the resultant $\mathrm{Res}(s_0, s_1)$. The resultant is an irreducible polynomial in the coefficients of $s_0, s_1$ (\cite[Chapter 8, Proposition-Definition 1.1]{GKZ}), which is homogeneous of degree $n(\lambda_0 + \lambda_1)$.

    Then, by \cite[(iv) Theorem 2.1.4.]{Fringuelli} we have an exact sequence
    \[
    \Zb \to \mathrm{Pic}(\Pc(\Lambdavec)) \to \mathrm{Pic}(\Hom_n(\Pb^1,\Pc(\lambda_0,\lambda_1))) \to 0
    \]
    where the first map is defined by $1 \mapsto 1 \cdot V(\mathrm{Res})$. As $\mathrm{Pic}(\Pc(\Lambdavec)) \iso \Zb$, we have that $$\mathrm{Pic}(\Hom_n(\Pb^1,\Pc(\lambda_0,\lambda_1))) \iso \Zb/\mathrm{Deg(Res)}\Zb = \Zb/((\lambda_0 + \lambda_1)n)\Zb.$$
    \end{proof} 
    
\bigskip

        
    \section{Stable cohomology of the Hom-stacks with weights}
    \label{sec:Coho}

    Recall that for a Deligne-Mumford stack $\mathcal{X}$ over $K$, we denote the $\ell$-adic cohomology group with rational coefficients by $H^i(\Xc;\mathbb{Q}_{\ell})$ (warning: this is not the \'etale cohomology with $\Ql$ coefficients, see e.g. \cite[Warning 3.2.1.9]{GL}, or any text on \'etale cohomology of schemes e.g. \cite{Milne}); by the same token a sheaf of $\Ql$ vector spaces is a $\mathbb{Z}_{\ell}$-sheaf $\mathcal{F}= (\mathcal{F}_n)$ and $$H^i(\Xc;\mathcal{F}):= \varprojlim H^i(\Xc;\mathcal{F}_n)\otimes_{\mathbb{Z}_{\ell}} \Ql$$ (similar to the notations set up in \cite[Section 19]{Milne}.) For brevity and convenience, and noting there is no scope of confusion since we are always working over rational coefficients, we will write $H^i(\mathcal{X})$ to stand for $H^i(\mathcal{X};\Ql)$.
    
   For $\lambdavec = (\lambda_0, \cdots, \lambda_N) \in \mathbb{Z}^{N+1}_{> 0}$ and positive integers $a, b$, we denote by $a\lambdavec +b$ the vector $(a\lambda_0 +b, a\lambda_1 +b, \cdots, a\lambda_N +b)$. Furthermore, given $\lambdavec = (\lambda_0,\cdots, \lambda_N)$ with $|\vec{\lambda}| \coloneqq \sum\limits_{i=0}^{N} \lambda_i$ and a vector space $V \coloneqq V_0\oplus \cdots \oplus V_N$ of dimension $|\vec{\lambda}|n +(N+1) -(N+1)g$, we denote by $\mathcal{P}(\oplus V_i, \lambdavec)$ the weighted projective stack where $\mathbb{G}_m$ acts on the direct summand $V_i$ of $V$ by weight $\lambda_i$. If $\mathcal{E}$ is a vector bundle on a space $\Xc$ such that $\mathcal{E} = \mathcal{E}_0\oplus \cdots\oplus \mathcal{E}_N$ then by $\mathcal{P}_{\Xc}(\oplus_i \mathcal{E}_i, \lambdavec)$ we denote the weighted projectivization of $\mathcal{E}$ where $\mathbb{G}_m$ acts on the sub-bundle $\mathcal{E}_i$ by weight $\lambda_i$. Note that when we write $\mathcal{P}_{\Xc} (\oplus_i \mathcal{E}_i, \lambdavec)$ we implicitly assume that $i$ runs from 0 to $N$ and that $\lambdavec \in \mathbb{Z}^{N+1}_{> 0}$.

    \subsection{Proof of Main Theorem~\ref{C-cohomThm}}
     
    Main Theorem \ref{C-cohomThm} (and in turn Theorem \ref{P-cohomThm}) is a direct consequence of \cite[Theorem 2]{Banerjee}. In particular, we construct an object over $\Delta S$ (called the \emph{symmetric simplicial category}, see \cite[Definition 2.10]{Banerjee}) in the category of (smooth proper) Deligne-Mumford stacks. 

    Now a degree $n$ morphism $C\to\mathcal{P}(\lambdavec)$ is equivalent to the following data:\begin{itemize}
        \item a line bundle $L$ of degree $n$ on $C$,
        \item an $(N+1)$-tuple $(s_0,\ldots, s_N)$ where $s_i\in H^0(C,L^{\otimes \lambda_i})$ 
        \item the sections $s_0,\ldots, s_N$ satisfy the condition that they have no common zeroes (also known as $\{s_0,\ldots, s_N\}$ is \emph{basepoint free}).
    \end{itemize} Then, $\mathrm{Hom}_n (C,\mathcal{P}(\lambdavec))$ is a Zariski open dense subset of $\mathrm{\widetilde{Hom}}_n (C,\mathcal{P}(\lambdavec))$ defined by \begin{gather*}
    \mathrm{\widetilde{Hom}}_n (C,\mathcal{P}(\lambdavec)) \coloneqq \{L,~[s_0:\ldots:s_N]: L \in \mathrm{Pic}^n(C),~s_i\in H^0(C,L^{\otimes \lambda_i}) \text{ for all }i\}.
    \end{gather*}  Note that $\mathrm{\widetilde{Hom}}_n (C,\mathcal{P}(\lambdavec))$ is isomorphic to the weighted projectivization of a vector bundle $\mathcal{E}$ on $\Pic^n(C)$ (assuming $n\geq 2g$) 
    whose fibre over $L$ is given by $V_L \coloneqq \oplus_i H^0(C,L^{\otimes \lambda_i}),$ where via the Riemann-Roch theorem we have  $H^0(C,L^{\otimes \lambda_i}) \cong \A^{\lambda_i n-g+1}$ and $\mathbb{G}_m$ acts by weight $\lambda_i$ on $H^0(C,L^{\otimes \lambda_i})$; fibre over $L\in \Pic^n C$ isomorphic to the weighted projective stack $\mathcal{P}(\oplus_i H^0(\mathbb{P}^1,\mathcal{O}(n)^{\otimes \lambda_i}),\lambdavec)$. 

    More precisely, if $$\nu: C\times \Pic^n C \to \Pic^n C$$ is the projection to the second factor and $P(n)$ denotes a Poincare bundle of degree $n$ on $C\times \Pic^n C$, then we let $\Gamma_{\lambda_i}(n)  \coloneqq \nu_* P(n)^{\otimes \lambda_i}$ (and when $\lambda_i=1$ we write $\Gamma (n)$). Then $\mathrm{\widetilde{Hom}}_n (C,\mathcal{P}(\lambdavec))$ is the weighted projectivization of $\oplus_i \Gamma_{\lambda_i}(n)$ where $\mathbb{G}_m$ acts on $\Gamma_{\lambda_i}(n)$ by weight $\lambda_i$.  By Theorem \ref{Thm:weightedprojectivebundle2} we know the cohomology of  $\mathrm{\widetilde{Hom}}_n (C,\mathcal{P}(\lambdavec))$- a key player in the proof of Theorem \ref{C-cohomThm} as we shall soon see. 

    Even though the exact cohomology of Theorem \ref{P-cohomThm} follows from stable cohomology of Theorem \ref{C-cohomThm} when we take the curve $C$ to be of genus $0$,  literature shows that not only is the genus 0 case worth proving in it's own right, in this manuscript it also portrays how the basic strategy of proving the higher genus case is essentially the same as that of the genus $0$ case; in other words if the reader understands the proof of Theorem \ref{P-cohomThm}, he knows the proof of Theorem \ref{C-cohomThm} as well. Keeping this in mind, we first prove Theorem \ref{P-cohomThm}.
    \begin{proof}[Proof of Theorem \ref{P-cohomThm}]
    First observe that $\mathrm{\widetilde{Hom}}_n (\Pb^1,\mathcal{P}(\lambdavec))$ 
    is a weighted projective stack isomorphic to $\mathcal{P}(\oplus_i H^0(\mathbb{P}^1,\mathcal{O}(n)^{\lambda_i},\lambdavec)$. Define a symmetric semisimplicial space (see \cite[Defintion 2.10]{Banerjee}) as follows. Let  \begin{align}
        \mathcal{T}_0 \coloneqq  \Pb^1\times \mathcal{P}(\oplus_i H^0(\mathbb{P}^1,\mathcal{O}(n\lambda_i-1)),\lambdavec)
    \end{align}\label{eqP1T0} and let \begin{align}
        \mathcal{T}_p \coloneqq  (\Pb^1)^{p+1}\times \mathcal{P}(\oplus_iH^0(\mathbb{P}^1,\mathcal{O}(n\lambda_i-p)),\lambdavec).
    \end{align}\label{eqP1Tp} 

    Letting $\mathcal{T}_{-1}$ denote $\mathrm{\widetilde{Hom}}_n (\Pb^1,\mathcal{P}(\lambdavec))$, we observe that there are natural \emph{face maps} (which are finite morphisms of smooth proper Deligne-Mumford stacks) $f_i: \mathcal{T}_p\to \mathcal{T}_{p-1}$ for all $p\geq 0$ given by \emph{adding a basepoint} from the $i^{th}$ factor i.e.

    \begin{gather}
        f_{i} : (\Pb^1)^{p+1} \times \mathcal{P}(\oplus_iH^0(\mathbb{P}^1,\mathcal{O}(n\lambda_i-(p+1)),\lambdavec) \to (\Pb^1)^{p} \times \mathcal{P}(\oplus_iH^0(\mathbb{P}^1,\mathcal{O}(n\lambda_i-p),\lambdavec) \nonumber\\
        \big([a_0:b_0],\ldots, [a_p:b_p]\big),\,\, [s_0:\ldots: s_N] \mapsto \nonumber \\ \big([a_0,b_0],\ldots, \widehat{[a_i:b_i]},\ldots [a_p:b_p]\big), \,\, [(b_ix-a_iy)s_0:\ldots: (b_ix-a_iy)s_N]\label{eq:facemap}
        \end{gather}

    \noindent where $\widehat{[a_i:b_i]}$ denotes removing the $i^{th}$ entry $[a_i:b_i]$. 
    In other words, the hypercover under consideration is the following:
        \begin{gather*}
        \cdots (\Pb^1)^3\times \mathcal{P}(\oplus_iH^0(\mathbb{P}^1,\mathcal{O}(n\lambda_i-3)),\lambdavec) \mathrel{\substack{\textstyle\rightarrow\\[-0.5ex]
                \textstyle\rightarrow \\[-0.5ex]
                \textstyle\rightarrow}}     (\Pb^1)^2 \times\mathcal{P}(\oplus_iH^0(\mathbb{P}^1,\mathcal{O}(n\lambda_i-2)),\lambdavec)\\ \rightrightarrows \Pb^1\times \mathcal{P}(\oplus_iH^0(\mathbb{P}^1,\mathcal{O}(n\lambda_i-1)),\lambdavec) \to \mathcal{P}(\oplus_iH^0(\mathbb{P}^1,\mathcal{O}(n\lambda_i)),\lambdavec)
        \end{gather*} with the unlabelled arrows denoting the face maps $f_i$.

        Let $\mathcal{Z}\subset \mathrm{\widetilde{Hom}}_n (\Pb^1,\mathcal{P}(\lambdavec))$ be the ``discriminant locus" given by \begin{align*}
        \mathcal{Z}=\{[s_0:\cdots: s_N]: s_i\in H^0(\Pb^1, \mathcal{O}(n)^{\otimes\lambda_i}) \text{ for all }i=0, \ldots, N, \\ s_0, \ldots, s_N \text{ have at least one common root} \}.
    \end{align*} Now we make the following observation (an almost immediate consequence of \cite[Theorem 7.1]{Conrad} or \cite{Deligne})):
\begin{claim}\label{CohomDescentClaim}
    The hypercovering $\mathcal{T}_{\bullet}\to \mathcal{Z}$ is universally of cohomological descent.
    \end{claim}
To prove it we show a more general lemma:
\begin{lem}\label{CohomDescentLemmaGen}
Let $\pi_{\bullet}: \mathcal{T}_{\bullet}\to \mathcal{Z}$ be a proper hypercovering in the category of Deligne-Mumford stacks, with the \'etale topology, such that for all $p$, and a morphism $$x: \mathrm{Spec}\,\, K \to \mathcal{T}_p$$ the kernel $\mathrm{Isom}_{\mathcal{T}_p} (x,x)\to \mathrm{Isom}_{\mathcal{Z}} (x,x)$ is trivial, then it is universally of cohomological descent.
\end{lem} 
\begin{proof}[Proof of Lemma \ref{CohomDescentLemmaGen}]
    The proof of this lemma follows immediately by noting that the claim is about \'etale sheaves on the base Deligne-Mumford stack $\mathcal{Z}$, so it suffices to check the claim \'etale locally. Now thanks to \cite[Lemma 100.6.2]{Stacks} the face maps are representable by algebraic spaces, so \'etale locally we can simply apply the result for schemes, which is precisely  \cite[Theorem 7.1]{Conrad}: that a proper hypercovering in the category of schemes is universally of cohomological descent.
\end{proof}

\begin{proof}[Proof of \ref{CohomDescentClaim}]
In the special case when $\mathcal{T}_{\bullet}$ is the specific simplicial space defined above (\ref{eqP1Tp}) and $\mathcal{Z}$ is the discriminant locus, the condition in the claim of the kernel being trivial is clearly satisfied, and thus the observation follows.
\end{proof}

Having settled that $T_{\bullet}\to \mathcal{Z}$ as a semisimplicial space is indeed of cohomological descent, by virtue of having the additional structure of a $\Delta_{inj} S$ object, it further satisfies the conditions of \cite[Theorem 2]{Banerjee} \footnote{In fact one can make a more general statement at no extra cost, simply by translating cohomological descent in the $\infty$-category of $\ell$-adic sheaves $\mathrm{Sh}_{\ell}(\mathcal{Z})$ (as defined, for example, by Gaitsgory-Lurie in \cite{GL}) over to the indexing category $\Delta S$, essentially giving an $\infty$-categorical version of \cite[Lemma 2.11]{Banerjee}, as follows. Given a $S_{\bullet}$ hypercovering $\pi_{\bullet}: \mathcal{T}_{\bullet} \to \mathcal{Z}$ satisfying the conditions of Lemma \ref{CohomDescentLemmaGen}, there is an equivalence of endofunctors $$\mathrm{id} \to \big(\mathrm{R}{\pi_{\bullet}}_* \pi^*_{\bullet} \otimes \mathrm{sgn}_{S_{\bullet}} \big)^{S_{\bullet}}$$ in the $\infty$-category $\mathrm{Fun}(\mathrm{Sh}_{\ell}(\mathcal{Z}), \mathrm{Sh}_{\ell}(\mathcal{Z}))$, where $S_{\bullet}$ denotes the symmetric simplicial group given by $S_n$ denoting the symmetric group on $(n+1)$ elements. However, we do not need the full power of this statement in for our immediate computation. The interested reader can refer to \cite{Banerjee2}.} and results in a second quadrant spectral sequence whose $E_1$ page reads as  \begin{gather*}
        E_1^{-p,q} = \begin{cases}
            H^q(\mathcal{P}(\oplus_i H^0(\Pb^1,\mathcal{O}(n\lambda_i)),\lambdavec))(0) & p=0,\\ H^{q-2N}(\Pb^1\times \mathcal{P}(\oplus_i H^0(\Pb^1,\mathcal{O}(n\lambda_i-1))(-1) & p=1,\\ H^0(\Pb^1)\otimes H^2(\Pb^1)\otimes H^{q-4N-2}(\mathcal{P}(\oplus_i H^0(\Pb^1,\mathcal{O}(n\lambda_i-2))(-2) & p=2,\\ 0 & \text{ otherwise},
        \end{cases}
    \end{gather*} with the differentials given by the alternating sum of the Gysin pushforwards induced by the face maps, which is what we shall compute now. 
    \begin{itemize}
        \item \textit{Computing $d_1^{1,q}: E_1^{-1,q} \to E_1^{0,q}$.}  
        
        For simplicity we denote the differential by $d_1^1$. Let  $$\iota: \mathcal{P}(\oplus_i H^0(\mathbb{P}^1,\mathcal{O}(n\lambda_i-1)),\lambdavec) \hookrightarrow\mathcal{P}(\oplus_i H^0(\mathbb{P}^1,\mathcal{O}(n\lambda_i)),\lambdavec)$$ denote the inclusion given by adding a basepoint. 
        
        \noindent Choose generators $\mathbf{1}\in H^0(\Pb^1)$ and $e\in H^2(\Pb^1)$, and let $h$ denote the hyperplane class in $\mathcal{P}(\oplus_i H^0(\mathbb{P}^1,\mathcal{O}(n\lambda_i),\lambdavec)$. Then we claim that: \begin{gather*}
            d_1^1={f_0}_*: H^{*-2N}(\Pb^1\times \mathcal{P}(\oplus_iH^0(\mathbb{P}^1,\mathcal{O}(n\lambda_i-1)),\lambdavec)) \to  \\ H^*(\mathcal{P}(\oplus_iH^0(\mathbb{P}^1,\mathcal{O}(n\lambda_i),\lambdavec))\\ \mathbf{1}\otimes \iota^*\alpha + e\otimes \iota^*\beta \mapsto \alpha h^N+ \beta h^{N+1}
        \end{gather*} is a map of $H^*(\mathcal{P}(\oplus_iH^0(\mathbb{P}^1,\mathcal{O}(n\lambda_i),\lambdavec))$-modules, where $$\alpha, \beta\in H^*(\mathcal{P}(\oplus_iH^0(\mathbb{P}^1,\mathcal{O}(n\lambda_i),\lambdavec).$$ To see this, first note that $$\iota^*: H^*(\mathcal{P}(\oplus_iH^0(\mathbb{P}^1,\mathcal{O}(n\lambda_i-1),\lambdavec)) \to H^*(\mathcal{P}(\oplus_iH^0(\mathbb{P}^1,\mathcal{O}(n\lambda_i),\lambdavec))$$ is a surjection; next, the image of the fundamental class $$[\Pb^1\times \mathcal{P}(\oplus_iH^0(\mathbb{P}^1,\mathcal{O}(n\lambda_i-1),\lambdavec)]\in H^0(\Pb^1\times \mathcal{P}(\oplus_iH^0(\mathbb{P}^1,\mathcal{O}(n\lambda_i-1),\lambdavec))$$ is the locus of elements in $\mathcal{P}(\oplus_iH^0(\mathbb{P}^1,\mathcal{O}(n\lambda_i),\lambdavec)$ that has a basepoint i.e. $\mathcal{Z}$, which is rationally equivalent, and thus cohomologous, to (a multiple of) $h^N$; and finally, for a fixed point $[a:b]\in \Pb^1$, the locus given by $$\{[s_0:\ldots: s_N]\in \mathcal{P}(\oplus_iH^0(\mathbb{P}^1,\mathcal{O}(n\lambda_i),\lambdavec): s_i([a:b])=0\}$$ is rationally equivalent, and in turn cohomologous, to (a multiple of) $h^{N+1}$. For the sake of simplicity we won't bother ourselves with the scalar multiples, which is fine because we're working over $\Qb$.
        The Gysin pushforward $d_1^1={f_0}_*$ surjects onto the ideal generated by $h^N$ in $H^*(\mathcal{P}(\oplus_iH^0(\mathbb{P}^1,\mathcal{O}(n\lambda_i,\lambdavec))$. Indeed, the preimage of $h^{N+i}$ is given by \begin{gather*}
        d_1^1(\mathbf{1}\otimes \iota^* h^i) =h^{N+i} = d_1^1(e\otimes \iota^*h^{i-1}) \text{ for } i\geq 1\\ d_1^1([\Pb^1\times \mathcal{P}(\oplus_iH^0(\mathbb{P}^1,\mathcal{O}(n\lambda_i-1),\lambdavec)) =h^N,
    \end{gather*} which shows that the image of $d_1^1$ is the ideal generated by $h^N$ in $H^*(\mathcal{P}(\oplus_iH^0(\mathbb{P}^1,\mathcal{O}(n\lambda_i),\lambdavec))$. The kernel of $d_1^1$ is given by elements of the form $\iota^*(\alpha)(h-e)$ for all $\alpha \in H^*(\mathcal{P}(\oplus_iH^0(\mathbb{P}^1,\mathcal{O}(n\lambda_i),\lambdavec))$. 

    The upshot is that on the $E_2$ page, for $p=0$ we have:\begin{gather}\label{eq3.7}
        E_2^{0,q} = \begin{cases}
            \Q(0) & q=2j, \,\, 0\leq j\leq 2(N-1)\\0 & \text{ otherwise.}
        \end{cases}
    \end{gather} 

    \item  \textit{Computing $d_1^{2,q}: E_1^{-2,q} \to E_1^{-1,q}$.} 
    For simplicity, we denote the differential by $d_1^2$. Like before, let $\iota: \mathcal{P}(\oplus_iH^0(\mathbb{P}^1,\mathcal{O}(n\lambda_i-1),\lambdavec) \hookrightarrow \mathcal{P}(\oplus_iH^0(\mathbb{P}^1,\mathcal{O}(n\lambda_i),\lambdavec)$ denote the inclusion given by adding a basepoint, and  let $h$ denote the hyperplane class in $\mathcal{P}(\oplus_iH^0(\mathbb{P}^1,\mathcal{O}(n\lambda_i-1),\lambdavec)$. Then the way we computed ${f_0}_*$ above works verbatim, and we have \begin{gather*}
        {f_0}_*:H^0(\Pb^1)\otimes H^2(\Pb^1)\otimes H^{*-2N-2}(\mathcal{P}(\oplus_iH^0(\mathbb{P}^1,\mathcal{O}(n\lambda_i-2)),\lambdavec)) \\\to H^{*}(\Pb^1\times \mathcal{P}(\oplus_iH^0(\mathbb{P}^1,\mathcal{O}(n\lambda_i-1),\lambdavec)\\
        \mathbf{1} \otimes e\otimes \alpha \mapsto e\otimes\alpha h^N 
    \end{gather*} and 
    \begin{gather*}
        {f_1}_*:H^0(\Pb^1)\otimes H^2(\Pb^1)\otimes H^{*-2N-2}(\mathcal{P}(\oplus_i H^0(\mathbb{P}^1,\mathcal{O}(n\lambda_i-2),\lambdavec)) \\ \to H^{*}(\Pb^1\times \mathcal{P}(\oplus_i H^0(\mathbb{P}^1,\mathcal{O}(n\lambda_i-1),\lambdavec))\\
        \mathbf{1} \otimes e\otimes \alpha \mapsto \mathbf{1}\otimes\alpha h^{N+1},
    \end{gather*} and therefore $$d_1^2(\mathbf{1}\otimes e \otimes \alpha)= \mathbf{1}\otimes \alpha h^{N+1}- e\otimes \alpha h^N.$$ Note that $d_1^2$ is injective, and the image is generated by $h^N$ in $$H^*(\Pb^1\times \mathcal{P}(\oplus_iH^0(\mathbb{P}^1,\mathcal{O}(n\lambda_i-1),\lambdavec)).$$ Consequently, on the $E_2$ page we have: \begin{flalign*}
    &E_2^{-1,q} = \begin{cases}
        \Ql(-N) & p=1, q=2j+2N+2, \,\, 0\leq j\leq 2(N-1)\\0 &\text{ otherwise}
    \end{cases},
    \\ & E_2^{-2,q} = 0,  \text{ for all } q.
    \end{flalign*} 
    \end{itemize}
    In effect on the $E_2$ page all differentials vanish; the spectral sequence degenerates and we obtain $$H^*(\mathrm{Hom}_n (\Pb^1,\mathcal{P}(\lambdavec));\Qb)\cong \frac{\Qb[h]}{h^N} \otimes \wedge\Qb\{t\}$$ where $h$ has cohomological degree $2$, and $t$ (which corresponds to $e-h\in Ker(d_1^1)$) has cohomological degree $2N+1$. Furthermore, over a field $\kappa$, with algebraic closure $\overline{\kappa}$, we have an isomorphism of $Gal (\overline{\kappa}/\kappa)$-representations:

    \begin{gather*}
        H^i_{\et}(\mathrm{Hom}_n (C,\mathcal{P}(\lambdavec));\Qb_{\ell}) = \begin{cases}
            {\Ql}(-j) & i=2j, 0\leq j\leq N-1\\ \Ql (-(j+1)) & i=2j+1, N\leq j\leq 2N-1\\ 0 & \text{ otherwise.} 
        \end{cases}
    \end{gather*}
    This completes the proof of Theorem \ref{P-cohomThm}.

    \end{proof}

    Following the same plan, we now proof Main Theorem \ref{C-cohomThm}.

    \begin{proof}[Proof of Theorem \ref{C-cohomThm}.] As already observed above, the space $\mathrm{Hom}_n (C,\mathcal{P}(\lambdavec))$ is open and dense in $\mathrm{\widetilde{Hom}}_n (C,\mathcal{P}(\lambdavec))$, and we have $$\mathrm{\widetilde{Hom}}_n (C,\mathcal{P}(\lambdavec))\cong \mathcal{P}_{\Pic^n C}(\oplus_i \nu_* P(n)^{\otimes \lambda_i},\lambdavec).$$ Let $\mathcal{Z}$ denote the discriminant locus i.e. the complement of  $\mathrm{Hom}_n (C,\mathcal{P}(\lambdavec))$ in $\mathrm{\widetilde{Hom}}_n (C,\mathcal{P}(\lambdavec))$. 

    Now we construct a hypercover that is equipped with the additional structure of a $\Delta S$ object in the category of Deligne-Mumford stacks, and which admits universal cohomological descent. 

To this end, we define spaces $\Tc_0$ and $\mathcal{I}_0$ as certain fibre products.
First, consider the following commutative diagram:

\[
\begin{tikzcd}[column sep=0]
  \Tc_0 \arrow[rr] \arrow[dr,dashed, swap,"\pi_0"] \arrow[dd,swap] &&
C\times \mathcal{P}_{\Pi_i\Pic^{\lambda_i n-1}C}(\oplus_i \nu_*P(\lambda_i n-1),\lambdavec) \arrow[dd] \arrow[dr,"A"] \\
  & \mathcal{P}_{\Pic^{n}C}(\oplus_i \nu_*P(n)^{\otimes \lambda_i},\lambdavec)  \arrow[rr] \arrow[dd]&&
  \mathcal{P}_{\Pi_i\Pic^{\lambda_i n}C}(\oplus_i \nu_*P(\lambda_i n),\lambdavec)  \arrow[dd] \\
\Ic_0 \arrow[rr,] \arrow[dr, dashed, "A"] && C\times \Pi_i \Pic^{\lambda_i n-1}C \arrow{dr}{A}\\
  & \Pic^n C\arrow[rr, "\otimes^{\lambdavec}"] && \Pi_i \Pic^{\lambda_i n}C 
\end{tikzcd}
\]
where the maps above are defined by:
  \begin{gather}
  \otimes^{\lambdavec}: \Pic^n \to \Pi_i \Pic^{\lambda_i n}C\nonumber \\
      L\mapsto (L^{\otimes \lambda_0},\cdots, L^{\otimes \lambda_N})
  \end{gather} henceforth often denoting $(L^{\otimes \lambda_0},\cdots, L^{\otimes \lambda_N})$ by $L^{\otimes \lambdavec}$;
  \begin{gather}\label{eq:otimeslambdavec}
      A: C\times \Pi_i \Pic^{\lambda_i n-1}C \to \Pi_i \Pic^{\lambda_i n}C \nonumber \\
      x, (L_0, \ldots, L_N)\mapsto L_0\otimes \mathcal{O}_C(x), \cdots, L_N\otimes \mathcal{O}_C(x)
  \end{gather} which is the map of `adding a point': $$\Ic_0 \coloneqq \Pic^n\times_{ \Pi_i \Pic^{\lambda_i n}C} (C\times \Pi_i \Pic^{\lambda_i n-1}C)$$ completes the commutative square at the `lower face' of the cube and we abuse notation and still denote the resulting `adding a point' map by $A: \Ic_0\to \Pic^n C$; the `upper face' of the cube consists of spaces which are essentially the space of global sections of suitable Poincare bundles, i.e. the vertical arrows all correspond to taking fibrewise global sections over the moduli of line bundles; and $\Tc_0$, quite like $\Ic_0$, is defined to complete the square on the upper face of the cube, and admits natural map to $\Ic_0$ so that each of the side faces are also naturally commutative squares. Whereas the `adding a point' map on the lower face of the cube adds points to line bundles, on the upper level they effectively add `basepoints' to global sections; in other words $$\pi_0: \Tc_0\to \mathcal{P}_{\Pic^n C}(\oplus_i \nu_* P(n)^{\otimes \lambda_i},\lambdavec)$$ simply captures the notion of adding a basepoint.
  
  We define spaces $\Tc_p$ for all $p\geq 0$ likewise. Consider  the following commutative diagram:

\adjustbox{scale=0.9,center}{%
\begin{tikzcd}[column sep=0]
  \Tc_p \arrow[rr] \arrow[dr,dashed, swap,"\pi_p"] \arrow[dd,swap] &&
C^{p+1}\times \mathcal{P}_{\Pi_i\Pic^{\lambda_i n-(p+1)}C}(\oplus_i \nu_*P(\lambda_i n-(p+1),\lambdavec) \arrow[dd] \arrow[dr,"f_1"] \\
  & \mathcal{P}_{\Pic^{n}C}(\oplus_i \nu_*P(n)^{\otimes \lambda_i},\lambdavec)  \arrow[rr] \arrow[dd]&&
  \mathcal{P}_{\Pi_i\Pic^{\lambda_i n}C}(\oplus_i \nu_*P(\lambda_i n),\lambdavec)  \arrow[dd,"h"] \\
\Ic_p \arrow[rr,] \arrow[dr, dashed, "A"] && C^{p+1}\times \Pi_i \Pic^{\lambda_i n-(p+1)}C \arrow{dr}{A}\\
  & \Pic^n C\arrow[rr, "\otimes^{\lambdavec}"] && \Pi_i \Pic^{\lambda_i n}C 
\end{tikzcd}
}
where $\otimes^{\lambdavec}$ is defined like before; $A$, by abusing notation, now denotes the addition of $(p+1)$ points to a line bundle i.e. \begin{gather}
      A: C^{p+1}\times \Pi_i \Pic^{\lambda_i n-(p+1)}C \to \Pi_i \Pic^{\lambda_i n}C \nonumber \\
      (x_0,\ldots, x_p), (L_0, \ldots, L_N)\mapsto L_0(\sum x_i), \cdots, L_N(\sum x_i);
  \end{gather}
  and moreover observe that we have commutative cubes of the following form for all $p$:
  
  \adjustbox{scale=0.9,center}{%
\begin{tikzcd}[column sep=0.5]
  \Tc_p \arrow[rr] \arrow[dr, swap,"f_i"] \arrow[dd,swap] &&
C^{p+1}\times \mathcal{P}_{\Pi_i\Pic^{\lambda_i n-(p+1)}C}(\oplus_i \nu_*P(\lambda_i n-(p+1),\lambdavec) \arrow[dd] \arrow[dr] \\
  & \Tc_{p-1}  \arrow[rr] \arrow[dd]&&
  C^{p}\times \mathcal{P}_{\Pi_i\Pic^{\lambda_i n-p}C}(\oplus_i \nu_*P(\lambda_i n-p,\lambdavec) \arrow[dd] \\
\Ic_p \arrow[rr,] \arrow[dr, dashed, "A"] && C^{p+1}\times \Pi_i \Pic^{\lambda_i n-(p+1)}C \arrow{dr}{A}\\
  & \Ic_{p-1}\arrow[rr, "\otimes^{\lambdavec}"] && C^p\times \Pi_i \Pic^{\lambda_i n-p}C 
\end{tikzcd}
} and the all the down-right arrows are maps corresponding to adding a point from the $i^{th}$ factor of $C^{p+1}$ i.e. 
\begin{gather*}
  A: C^{p+1} \times \Pi_i \Pic^{\lambda_i n-(p+1)}C \to C^p\times \Pi_i \Pic^{\lambda_i n-p}C \\
  (x_0,\ldots, x_p), (L_0,\ldots, L_N)\mapsto\\ (x_0,\ldots,\widehat{x_i},\ldots, x_p), (L_0(x_i),\ldots, L_N(x_i))
\end{gather*} where $\widehat{x_i}$ implies we delete the $i^{th}$ entry; most importantly, we obtain $f_i:\Tc_p\to \Tc_{p-1}$ as the $i^{th}$ face map of the semisimplicial space $\Tc_{\bullet}$. 

It is easy to see that $\Ic_p \cong C^{p+1}\times \Pic^nC$ for all $p\geq 0$: for $p=0$ the map \begin{gather*}
  \Pic^nC{\times_i \Pi_i \Pic^{\lambda_i n}C}. C\times \Pi_i \Pic^{\lambda_i n-1}C \to C\times \Pic^{n}C\\
  L, (x,(L_0,\cdots, L_N))\mapsto (x,L)
\end{gather*} has a natural inverse thereby giving an isomorphism; for higher values of $p$ the observation follows likewise. Along those lines, it is not hard to show for $p\leq n-2g+1$ that $\Tc_p\cong C^{p+1}\times \mathcal{P}_{\Pic^n C} (\oplus_i \mathcal{E}_i, \lambda_i)$ (the isomorphism is not canonical) for some sub-vector budnles $\mathcal{E}_i\subset \nu_*P(n)^{\otimes \lambda_i}$ for each $i$, of rank $\lambda_i n-g+1 -(p+1)$, in the following way. For any line bundle $L$ on $C$ and any  effective divisor $D$ on $C$ there is a natural injection \begin{gather*}
    H^0(C,L)\to H^0(C,L(D))\\ s\mapsto s(D)
\end{gather*} coming from the short exact sequence of the corresponding locally free sheaves $$0\to L\to L(D)\to \mathcal{O}_D\to 0;$$ by definition we have \begin{gather*}
    \Tc_p: =\{[s_0:\ldots:s_N], ((x_0,\cdots, x_p),[\widetilde{s_0}:\ldots: \widetilde{s_N})): s_i\in H^0(C,L^{\lambda_i}) \text{ where } L\in \Pic^n C,\\ \widetilde{s_i}\in H^0(C,L_i) \text{ where } L_i\in \Pic^{\lambda_i n-(p+1)}C,\\ \widetilde{s_i}(\sum x_j)=t^{\lambda_i} s_i \text{ for some } t\in \mathbb{G}_m \text{ and for all }i\}.
\end{gather*} Now fix a general unordered $(p+1)$ subset $\{z_0,\cdots, z_p\}\in C$ and define \begin{gather*}
    M_p:= \{[\widetilde{s_0}:\ldots: \widetilde{s_N}]: \widetilde{s_i}\in H^0(C,L_i) \text{ where } L_i\in \Pic^{\lambda_i n-(p+1)}C, \text{ there exists } t\in \mathbb{G}_m\\ \text{ such that for each }i,\,\,\, \widetilde{s_i}(\sum_j c_j)=t^{\lambda_i} s_i,\\ \text{ where } s_i\in H^0(C,L^{\lambda_i}), \,\,\,\,  L\in \Pic^n C\}
\end{gather*}
Then one can define a bijective (easy to check) morphism that only depends on the choice of the points $z_0,\ldots, z_p \in C$: \begin{gather}
    C^{p+1} \times M_p\to \Tc_p \nonumber \\ (x_0,\ldots, x_p), [\widetilde{s_0}:\ldots: \widetilde{s_N}]\mapsto\nonumber\\ [\widetilde{s_0}(\sum_j x_j):\ldots: \widetilde{s_N}(\sum_j x_j)], ((x_0,\cdots, x_p), [\widetilde{s_0}:\ldots: \widetilde{s_N}])
\end{gather} where the only non-trivial part is to check that the map is well-defined and that fact follows from the observation that degree $0$ line bundles always have $\lambda^{th}$ roots for all positive integers $\lambda$. Now $M_p$ is naturally a fibre bundle over $\Pic^n C$, in fact it is a weighted projectivization of the sub-vector bundle of $\oplus \nu_* P(n)^{\otimes \lambda_i}$ whose fibres are given by $(N+1)$-tuples of sections in $\oplus \nu_* P(n)^{\otimes \lambda_i}$ having common zeroes at $z_0,\ldots, z_p$; we denote that sub-vector bundle by $\oplus \nu_* P(n)^{\otimes \lambda_i}_{z_0+\cdots+z_p}$. In other words $$M_p\cong \mathcal{P}_{\Pic^n C}(\oplus_i \nu_* P(n)^{\otimes \lambda_i}_{z_0+\cdots+z_p},\lambdavec),$$ and for all $0\leq p\leq n-2g+1$ we abuse notation and denote the fibre bundle by $\rho$ i.e. $$\rho:M_p\to \Pic^n C.$$

   
   Deviating from the norm, we set $\Tc_{-1} \coloneqq  \mathrm{\widetilde{Hom}}_n (C,\mathcal{P}(\lambdavec))\cong  \mathcal{P}_{\Pic^n C}(\oplus_i \nu_* P(n)^{\otimes \lambda_i},\lambdavec)$) (whereas the norm would be to set $\Tc_{-1}$  as $\mathcal{Z}$- this a minor deviation just to simplify the notations involved in the subsequent computation). 
   From \cite[Definition 2.1]{Banerjee}, $\Tc_{\bullet}$ satisfies the conditions of being a $\Delta S$ object. We therefore use \cite[Theorem 1.2]{Banerjee} to get a second quadrant spectral sequence that reads as: 
    \begin{gather*}
        E_1^{-p,q}= H^{q-2pN}\big(\Tc_p\otimes \mathit{sgn}_{p+1}\big)^{{S}_{(p+1)}}(-(p+1)N) \implies H^{q+p}(\mathrm{Hom}_n (C,\mathcal{P}(\lambdavec)).
    \end{gather*}
     where $\mathit{sgn}_{p+1}$ denote the sign action of the symmetric group $S_{p+1}$ on the $(p+1)$ factors of $T_p$ by permutation; and the differentials of this spectral sequence are given by the alternating sum of the Gysin pushforwards induced by the face maps. So we split the rest of the proof into the following parts:
     \begin{enumerate}
        \item \textit{Computing the $E_1$ terms.} To this end note that a complete understanding of $$H^*(\mathcal{P}_{\Pic^n C}(\oplus_i \nu_* P(n)^{\otimes \lambda_i}_{z_0+\cdots+z_p},\lambdavec))$$ (at least in a stable range) gives us full knowledge of the $E_1$ terms (in that range).  The Chern classes of $\nu_* P(n)^{\otimes \lambda_i}$ can be computed for example, directly using Grothendieck-Riemann-Roch, or via ad-hoc methods: $$c_i(\mathcal{E}_{\lambda_i}) = (-1)^i {\theta^i \over i!}\,\,\, i=0, \ldots, g$$
     where $\theta$ is the fundamental class of the theta divisor (several proofs are available in \cite[Sections 4, 5, Chapter VII and Section 1, Chapter VIII]{ACGH}). Using the Whitney sum formula we can express the twisted Chern classes in terms of $\theta$.
     
     
     \noindent In turn, let $N_0 \coloneqq  (n-g+1)(N+1)$, which is the dimension of the fibres of $\mathcal{E}\to \Pic^n(C)$,  and let $h$ denote the relative hyperplane class i.e. $h= c_1(\mathcal{O}_{\rho}(1)) \in H^2(\mathcal{P}_{\Pic^n C}(\oplus_i \nu_* P(n)^{\otimes \lambda_i}, \lambdavec)$, then $H^*(\mathcal{P}_{\Pic^n C}(\oplus_i \nu_* P(n)^{\otimes \lambda_i}, \lambdavec)$, which is an algebra on $H^*(\Pic^n (C))\cong\wedge(H^1(C))$, is given by (using \eqref{eqn:projectivebundleformula_Coho}): \begin{gather}\label{eq3.9}
     H^*(\mathcal{P}_{\Pic^n C}(\oplus_i \nu_* P(n)^{\otimes \lambda_i}, \lambdavec)\nonumber \\\cong \nonumber \\{H^*(\Pic^n (C))[h]\over {h^{N_0}+\rho^*c^{\vec{\eta}}_1(\oplus_i \nu_* P(n)^{\otimes \lambda_i},\lambdavec) h^{N_0-1}+ \ldots + \rho^*c^{\vec{\eta}}_g(\oplus_i \nu_* P(n)^{\otimes \lambda_i},\lambdavec)h^{N_0-g}}}.
    \end{gather}    
     Let $p$ be such that $n-p\geq 2g$ and let $N_p \coloneqq  (n-p-g+1)(N+1) = N_0-p(N+1)$, the dimension of the fibres of $\mathcal{P}_{\Pic^n C}(\oplus_i \nu_* P(n)^{\otimes \lambda_i}_{z_0+\cdots+z_p},\lambdavec)\to \Pic^{n}(C)$, then combining \eqref{eqn:projectivebundleformula_Coho} and \eqref{eq3.9} we have a complete description of the $E_1$ terms of the spectral sequence above. We remark here that since $n-p\geq 2g$,  we have that $N_p - g= (n-p-g+1)(N+1)-g\geq  N$. This remark will be useful later.
        
        \item \textit{Computing the differentials $d_1^p: E_1^{-p,*}\to E_1^{-(p-1),*+2N}$.}
            
            Following previously introduced notations, let $h= c_1(\mathcal{O}_{\rho_{n}}(1))$, and for all $p$ satisfying $n-p\geq 2g$, let $$\iota: \mathcal{P}_{\Pic^n C} (\oplus_i \nu_* P(n)^{\otimes \lambda_i}_{z_0+\cdots+ z_{p-1}},\lambdavec)\to \mathcal{P}_{\Pic^n C} (\oplus_i \nu_* P(n)^{\otimes \lambda_i}_{z_0+\cdots+ z_p},\lambdavec)$$ denote the closed embedding induced by adding a basepoint $x$ (an abuse of notation that won't cause any confusion down the way). Note that $\iota$ is fibrewise linear embedding, up to translation of $\Pic^n(C)$ by $x$. Finally, let $e\in H^2(C)$ be the class of a point, $\mathbb{1}$ the fundamental class of $C$, and let $\gamma_1,\ldots, \gamma_{2g}$ be the standard basis of $H^1(C)$ and because $H^*(\Pic^n(C)) \cong \wedge H^1(C)$, let $\overline{\gamma_1},\ldots, \overline{\gamma_{2g}}$ be the image of $\gamma_1,\ldots, \gamma_{2g}$ under the aforementioned isomorphism. 
            
            First, we observe that \begin{align*}
                d_1^{1}: H^*(C\times \mathcal{P}_{\Pic^n C} (\oplus_i \nu_* P(n)^{\otimes \lambda_i}_{z_0},\lambdavec)) \to H^*(\mathcal{P}_{\Pic^n C} (\oplus_i \nu_* P(n)^{\otimes \lambda_i},\lambdavec))\\ [C\times \mathcal{P}_{\Pic^n C} (\oplus_i \nu_* P(n)^{\otimes \lambda_i}_{z_0},\lambdavec)] \mapsto h^N\\ e\mapsto h^{N+1}\\ \gamma_i\mapsto \overline{\gamma_i}h^N, & \text{ for all }i.
            \end{align*} is a map of $H^*(\mathcal{P}_{\Pic^n C} (\oplus_i \nu_* P(n)^{\otimes \lambda_i},\lambdavec))$-modules, and in turn $$\iota^*\alpha + e \iota^*\beta + \sum_{i=1}^{2g} \gamma_i \iota^* \gamma_i \xmapsto{d_1^1} \alpha h^N+ \beta h^{N+1}+ \sum_{i=1}^{2g}  \overline{\gamma_i}\gamma_i h^N,$$ where $\alpha, \beta, \gamma_1,\ldots, \gamma_{2g} \in  H^*(\mathcal{P}_{\Pic^n C} (\oplus_i \nu_* P(n)^{\otimes \lambda_i},\lambdavec))$.
            Indeed, the justification for the formula for $d_1^1$ in the previous case of $C=\Pb^1$ holds almost verbatim here. We know $$\iota^*: H^*(\mathcal{P}_{\Pic^n C} (\oplus_i \nu_* P(n)^{\otimes \lambda_i},\lambdavec)) \to H^*(\mathcal{P}_{\Pic^n C} (\oplus_i \nu_* P(n)^{\otimes \lambda_i}_{z_0},\lambdavec))$$ is a surjection; next, for a fixed point $x\in C$, the image $t_x^r(\mathcal{P}_{\Pic^n C} (\oplus_i \nu_* P(n)^{\otimes \lambda_i}_{z_0},\lambdavec)$) is rationally equivalent, and in turn cohomologous, to (a multiple of) $h^{N+1}$, and finally, that the image of the fundamental class $[C\times \mathcal{P}_{\Pic^n C} (\oplus_i \nu_* P(n)^{\otimes \lambda_i}_{z_0},\lambdavec)]\in H^0(C \times \mathcal{P}_{\Pic^n C} (\oplus_i \nu_* P(n)^{\otimes \lambda_i}_{z_0},\lambdavec))$ is rationally equivalent, and thus cohomologous, to (a multiple of) $h^N$, can be seen as in the following way. 
            Recall that a Poincaré bundle $P(n)$ is $\nu$-relatively very ample for all $n\geq 2g-1$, which in turn induces a relative embedding of $C\times \Pic^n(C)\xrightarrow{i_n}\Pb (\nu_* P(n))$ over $\Pic^n(C)$ and we have a natural sequence of maps over $\Pic^n C$
            \[
            \begin{tikzcd}
    C\times \Pic^n C \arrow[r, hook] \arrow[dr]
    & \Pb(\nu_* P(n)) \arrow[d] \arrow[r]& \mathcal{P}_{\Pic^n C}(\oplus_i \nu_* P(n)^{\lambda_i},\lambdavec) \arrow[dl]\\
    &\Pic^n C \end{tikzcd}
            \] and we continue to denote the composition mapping $C\times \Pic^n C$ to $\mathcal{P}_{\Pic^n C}(\oplus_i \nu_* P(n)^{\lambda_i},\lambdavec)$ over $\Pic^n C$ by $i_n$.
        This makes $i_n(C\times \Pic^n(C))$ in $\mathcal{P}_{\Pic^n C} (\oplus_i \nu_* P(n)^{\otimes \lambda_i},\lambdavec)$ homologous to (a scalar multiple of) the (relative, over the base $\Pic^n C$) Poincaré dual of $h \in H^2(\mathcal{P}_{\Pic^n C} (\oplus_i \nu_* P(n)^{\otimes \lambda_i},\lambdavec))$.  In turn, the image of the $[C\times \mathcal{P}_{\Pic^n C} (\oplus_i \nu_* P(n)^{\otimes \lambda_i}_{z_0},\lambdavec)]$ under the Gysin map ${f_0}_*$ is given by \begin{align*}
                {f_0}_*\big([C\times \mathcal{P}_{\Pic^n C} (\oplus_i \nu_* P(n)^{\otimes \lambda_i}_{z_0},\lambdavec)]\big) =h^{N+1}\frown i_n(C\times \Pic^n(C))\\ =h^N.
            \end{align*}
            Yet again, for the sake of simplicity we won't bother ourselves with the scalar multiples, which is fine because we're working over $\Q$.
            Noting that $$\overline{\gamma_i}(e-h)+ h(\gamma_i-\overline{\gamma_i}) = \gamma_i h- e \overline{\gamma_i},$$ it is now easy to check that the kernel of $d_1^1$ is given by: \begin{align*}
                H^*(\mathcal{P}_{\Pic^n C} (\oplus_i \nu_* P(n)^{\otimes \lambda_i}_{z_0},\lambdavec))(e-h) [2N]\\ \bigoplus    H^*(\mathcal{P}_{\Pic^n C} (\oplus_i \nu_* P(n)^{\otimes \lambda_i}_{z_0},\lambdavec))(\gamma_i-\overline{\gamma_i})[2N], &&(i= 1, \ldots, 2g)
            \end{align*} where $[2N]$ denotes a shift in the cohomological degree by $2N$, and which is viewed as a $\iota^* H^*(\mathcal{P}_{\Pic^n C} (\oplus_i \nu_* P(n)^{\otimes \lambda_i},\lambdavec))\cong H^*(\mathcal{P}_{\Pic^n C} (\oplus_i \nu_* P(n)^{\otimes \lambda_i}_{z_0},\lambdavec))$-module. The cokernel of $d^1_1$, which forms $E_2^{0,*}$ is given by  $$H^*(\Pic^n(C))[h]\over h^N$$(where note that as remarked before $r<N_0-g$, see \eqref{eq3.9}).
            
            Now we work out the differential for $p=2$ by computing the Gysin pushfowards by each of the face maps: \begin{gather*}
                {f_0}_*(\mathbb{1}\otimes e) = e h^N, \,\,\,\,  {f_1}_*(\mathbb{1}\otimes e) = h^{N+1}  \implies d_1^2 (\mathbb{1}\otimes e) =  (e-h) h^{N}, \\     {f_0}_*(e\otimes \gamma_i) = \gamma_i h^{N+1}, \,\,\,\,     {f_1}_*(e\otimes \gamma_i) = e \overline{\gamma_i}h^N \implies d^2_1(e\otimes \gamma_i)= (\gamma_ih- e\overline{\gamma_i}) h^N,\\
                {f_0}_*(\mathbb{1}\otimes \gamma_i) = \gamma_i h^N, \,\,\,\,    {f_1}_*(\mathbb{1}\otimes \gamma_i) = \overline{\gamma_i} h^N \implies d^2_1(\mathbb{1}\otimes \gamma_i) =  (\gamma_i -  \overline{\gamma_i}) h^N , \\  d_1^2(\gamma_i\gamma_j)=0,
            \end{gather*} where the last equality follows form the fact that on $\Sym^p H^1(C)$ for $p\geq 2$, the alternating sum of face maps is, by definition, $0$.
            Recalling our earlier remark that $N<N_1-g$, we see that the $E_2^{-1,*}$ terms, as an $H^*(\mathcal{P}_{\Pic^n C} (\oplus_i \nu_* P(n)^{\otimes \lambda_i}_{z_0+z_1},\lambdavec))$-module, are given by:
            \begin{align*}
                {H^*(\Pic^n(C);\Q(-N))[h]\over h^N}(e-h)[2N] \\ \bigoplus_{1\leq i\leq 2g} {H^*(\Pic^n(C);\Q(-r))[h]\over h^N}(\gamma_i-\overline{\gamma_i})[2N].
            \end{align*}Whereas the kernel of $d_1^2$ is generated by exactly what one expects: as a $H^*(\mathcal{P}_{\Pic^n C} (\oplus_i \nu_* P(n)^{\otimes \lambda_i}_{z_0+z_1},\lambdavec))$-module, we have  \begin{align*}
                \mathrm{Ker}(d_1^2)=  \bigoplus_{1\leq i\leq 2g} H^*(\mathcal{P}_{\Pic^n C} (\oplus_i \nu_* P(n)^{\otimes \lambda_i}_{z_0+z_1},\lambdavec))\big(e\otimes \gamma_i  - 1\otimes \gamma_i h +\mathbb{1}\otimes e \overline{\gamma_i}\big)[4N]\\ \bigoplus_{1\leq i,j\leq 2g} H^*(\mathcal{P}_{\Pic^n C} (\oplus_i \nu_* P(n)^{\otimes \lambda_i}_{z_0+z_1},\lambdavec))(\gamma_i\gamma_j)[4N].
            \end{align*}
            For $p= 3$ we have $d_1^{3}: E_1^{-3,*}\to E_1^{-2,*}$ given by:
            \begin{gather*}
                d_1^3 (\mathbb{1}\otimes e \otimes \gamma_i) =   e\otimes \gamma_i h^N- \mathbb{1}\otimes \gamma_i h^{N+1} +\mathbb{1}\otimes e \overline{\gamma_i} h^{N}\impliedby \begin{cases}
                    {f_0}_*(\mathbb{1}\otimes e \otimes \gamma_i) = e\otimes \gamma_i h^N,\\ {f_1}_*(\mathbb{1}\otimes e\otimes \gamma_i) = \mathbb{1}\otimes \gamma_i h^{N+1}\\ {f_2}_*(\mathbb{1}\otimes e\otimes \gamma_i) = \mathbb{1}\otimes e \overline{\gamma_i} h^{N}
                \end{cases}\\
                d_1^3(e\otimes \gamma_i\gamma_j) = \gamma_i\gamma_j h^{N+1}, \\
                d_1^3(\mathbb{1}\otimes \gamma_i\gamma_j) = \gamma_i\gamma_j h^N, \\    d_1^3(\gamma_i\gamma_j\gamma_k)=0,
            \end{gather*} where, for the last three equalities, recall again that on $\Sym^p H^1(C)$ for $p\geq 2$, the alternating sum of face maps is, by definition, $0$. Therefore the $E_1^{-2,*}$ terms defined by $\mathrm{Ker}(d^2_1)/ \mathrm{Coker}(d^3_1)$ is given by: 
            \begin{align*}
                \bigoplus_{1\leq i\leq 2g}{H^*(\Pic^n(C);\Q(-2N))[h]\over h^N}\big(e\otimes \gamma_i  - 1\otimes \gamma_i h +\mathbb{1}\otimes e \overline{\gamma_i}\big)[4N] \\ \bigoplus_{1\leq i,j\leq 2g}{H^*(\Pic^n(C);\Q(-2N))[h]\over h^N}(\gamma_i\gamma_j)[4N].
            \end{align*} 
            The formula for the differentials in the case of $p\geq 3$ mimics that of $p=3$, and we have:
            
            \begin{gather*}
                \mathbb{1}\otimes e\otimes c^{\vec{\eta}}_1\ldots c^{\vec{\eta}}_{p-2} \mapsto \Big((e\otimes c^{\vec{\eta}}_1\ldots c^{\vec{\eta}}_{p-2}) - (\mathbb{1} \otimes c_1\ldots c_{p-2}) h\Big)h^N,\\
                e\otimes c^{\vec{\eta}}_1\ldots c^{\vec{\eta}}_{p-1} \mapsto c^{\vec{\eta}}_1\ldots c^{\vec{\eta}}_{p-1} h^{N+1},\\  \mathbb{1} \otimes c^{\vec{\eta}}_1\ldots c^{\vec{\eta}}_{p-1} \mapsto c^{\vec{\eta}}_1\ldots c^{\vec{\eta}}_{p-1} h^{r} \\ c^{\vec{\eta}}_1\ldots c^{\vec{\eta}}_p \mapsto 0
            \end{gather*}
            It is now easy to check that \begin{gather*}
                \mathrm{Ker}(d_1^p)/\mathrm{Coker} (d_1^{p+1}) =\\ \bigoplus_{1\leq i\leq 2g}{H^*(\Pic^n(C);\Q(-pN))[h]\over h^N}\big(e\otimes c^{\vec{\eta}}_1\ldots c^{\vec{\eta}}_{p-1} - \mathbb{1} \otimes c^{\vec{\eta}}_1\ldots c^{\vec{\eta}}_{p-1} \big)[2pN]\\ \bigoplus_{1\leq i,j\leq 2g}{H^*(\Pic^n(C);\Qlb(-pN))[h]\over h^N}(c^{\vec{\eta}}_1\ldots c^{\vec{\eta}}_p)[2pN].
            \end{gather*}
            
        \end{enumerate}
        
        Now we are left with analysing the resulting $E_2$ page. That the differentials on the $E_2$ page vanish for $p\leq n-2g$ follow simply from weight considerations- the space $\Tc_{\bullet}$ consists of smooth projective varieties and thus their $n^{th}$ cohomology is pure of weight $n$. Now observe the following: we have an equality \begin{gather*}
            \mathrm{R}\Gamma_c(\mathcal{P}_{\Pic^n C} (\oplus_i \nu_* P(n)^{\otimes \lambda_i},\lambdavec),C^{\bullet}(\underline{\Ql}_{\mathcal{P}_{\Pic^n C} (\oplus_i \nu_* P(n)^{\otimes \lambda_i},\lambdavec)})) =\\ \mathrm{R}\Gamma_c (\mathcal{P}_{\Pic^n C} (\oplus_i \nu_* P(n)^{\otimes \lambda_i},\lambdavec), j_{!}\underline{\Ql}_{\mathrm{Hom}_n (C,\mathcal{P}(\lambdavec))}) 
        \end{gather*}in the derived category of constructible sheaves over $\mathcal{P}_{\Pic^n C} (\oplus_i \nu_* P(n)^{\otimes \lambda_i},\lambdavec)$ where $C^{\bullet}(\Q_{\mathcal{P}_{\Pic^n C} (\oplus_i \nu_* P(n)^{\otimes \lambda_i},\lambdavec)})$ denotes the complex \begin{align}
0\to j_{!}j^*\underline{\Ql}_{\Tc_{-1}}\to \underline{\Ql}_{\Tc_{-1}}\to \pi_{0_*}\pi_0^*\underline{\Ql}_{\Tc_{-1}} \to (\pi_{1_*}\pi_1^*\underline{\Ql}_{\Tc_{-1}}\otimes \mathit{sgn}_{2})^{{S}_{2}} \cdots \to \nonumber \\ \cdots \to  (\pi_{p_*}\pi_p^* \underline{\Ql}_{\Tc_{-1}}\otimes \mathit{sgn}_{p+1})^{{S}_{p+1}}\to \cdots 
    \end{align} 
    on the other hand, for any $m\in \mathbb{N}$ we have \begin{gather*}
        \mathrm{R}^i\Gamma_c(\mathcal{P}_{\Pic^n C} (\oplus_i \nu_* P(n)^{\otimes \lambda_i},\lambdavec), C^{\bullet}(\underline{\Ql}_{\mathcal{P}_{\Pic^n C} (\oplus_i \nu_* P(n)^{\otimes \lambda_i},\lambdavec)}))\cong \\ \mathrm{R}^i\Gamma_c(\mathcal{P}_{\Pic^n C} (\oplus_i \nu_* P(n)^{\otimes \lambda_i},\lambdavec), C^{\bullet}(\underline{\Ql}_{\mathcal{P}_{\Pic^n C} (\oplus_i \nu_* P(n)^{\otimes \lambda_i},\lambdavec)})/\tau_{\geq m} C^{\bullet}(\underline{\Ql}_{\mathcal{P}_{\Pic^n C} (\oplus_i \nu_* P(n)^{\otimes \lambda_i},\lambdavec)})
    \end{gather*} for all $i\geq 2(m+1)-2N$, where $\tau_{\geq m} C^{\bullet}(\Q_{\mathcal{P}_{\Pic^n C} (\oplus_i \nu_* P(n)^{\otimes \lambda_i},\lambdavec)})$ denotes the truncated complex up to the $(N-1)$ term and this is because $\tau_{\geq m} C^{\bullet}(\Q_{\mathcal{P}_{\Pic^n C} (\oplus_i \nu_* P(n)^{\otimes \lambda_i},\lambdavec)})$ is supported on complex codimension $m$ in $\mathcal{P}_{\Pic^n C} (\oplus_i \nu_* P(n)^{\otimes \lambda_i},\lambdavec)$. Therefore the cohomology of $\mathrm{Hom}_n (C,\mathcal{P}(\lambdavec))$ up to degree $n-2g$ is solely dictated by the $E_2$ page.
    
            To this end, let $$t:= (e-h)$$ which has degree $(-1, 2N+2)$ and let $$\alpha_i:= \gamma_i-\overline{\gamma_i},\,\,\,\,\, i=1,\ldots, 2g$$ which has degree $(-1,2N+1)$. Clearly for $3\leq p\leq n-2g$, the element $t\alpha_{i_1}\ldots \alpha_{i_p}$, which is of degree $(-(p+1), 2N+2+p(2N+1))$, when expanded, gives us  \begin{flalign*}
                t\alpha_{i_1}\ldots \alpha_{i_p}\\ &= (e-h)(c_{i_1}-\overline{c_{i_1}})\ldots (c_{i_p}-\overline{c_{i_p}}) \\&= (e-h)\prod_{j=1}^{p}c_{i_j} +\Big\{ \text{lower order terms as a polynomial on } c_{i_1},\ldots, c_{i_p}\Big\} \\ & = (e-h)\prod_{j=1}^{p}c_{i_j}
            \end{flalign*} because the lower order terms are all $0$ in $$\big(H^2(C)\oplus H^0(C)\big)\bigotimes \Sym^p H^1(C)\otimes H^*(\Pic^{n-(p+1)}(C))[h]/h^N,$$ thanks to the alternating action of ${S}_{p+1}$. Whereas  $\alpha_{i_1}\ldots \alpha_{i_{p+1}}$, which is of degree $(-(p+1), (p+1)(2N+1))$,  when expanded, gives us
            \begin{flalign*}
                \alpha_{i_1}\ldots \alpha_{i_{p+1}}\\ &= (c_{i_1}-\overline{c_{i_1}})\ldots (c_{i_{p+1}}-\overline{c_{i_{p+1}}}) \\&= \prod_{j=1}^{p+1}c_{i_j} +\Big\{ \text{lower order terms as a polynomial on } c_{i_1},\ldots, c_{i_{p+1}}\Big\} \\ & = \prod_{j=1}^{p+1}c_{i_j}
            \end{flalign*}  because again, the lower order terms are all $0$ for the exact same reason cited above.
            
            Now as for $p=2$,  we have \begin{align*}
                t\alpha_i = (e-h)(\gamma_i-\overline{\gamma_i}) = e\gamma_i- \gamma_ih+e\overline{\gamma_i}+ h\overline{\gamma_i} \\ =e\gamma_i- \gamma_ih+e\overline{\gamma_i}
            \end{align*} because the alternating action of ${S_2}$ kills $H^0(C^2)\otimes H^*(\mathcal{P}_{\Pic^n C} (\oplus_i \nu_* P(n)^{\otimes \lambda_i},\lambdavec))$, and in turn, $h\overline{\gamma_i}$.
            This give us the algebra structure on the $E_2$ page for $p\leq n-2g$ and thus completes the proof of Theorem \ref{C-cohomThm}.
    
    \end{proof}

    \bigskip


    \section{Arithmetic moduli of generalized elliptic surfaces with prescribed structures}\label{sec:Gen_Ell}

    In this section, we prove new sharp enumerations on the number of generalized elliptic fibrations over $C=\Pb^1_{\Fb_q}$ with prescribed level structures or multiple marked points by applying the exact \'etale cohomology Theorem \ref{P-cohomThm} (followed by the exact point count Theorem \ref{Exac_Count}) to the relevant moduli stacks formulated as Hom-stacks similar to \cite[\S 3]{HP}. In all our applications, results over higher genus $C$ where we acquire the corresponding stable \'etale cohomology followed by stable point counts as in Corollary \ref{Ell_Count} is straightforward.


    \subsection*{Arithmetic moduli of generalized elliptic fibrations over \texorpdfstring{$C$}{C} with prescribed level structures}\label{subsec:ptcount_t}

    We enumerate the number of generalized elliptic curves over global function fields with prescribed level structures by first extending the notion of (nonsingular) elliptic curves that admits desired level structures. By the work of Deligne and Rapoport \cite{DR} (summarized in \cite[\S 2]{Conrad2} and also in \cite[\S 2]{Niles}), we consider the generalized elliptic curves over $C_K$ with $[\Gamma]$--level structures over a field $K$ (focusing on $K=\Fb_q$). 

    Recall that a level structure $[\Gamma_1(m)]$ on an elliptic curve $E$ is a choice of point $P \in E$ of exact order $m$ in the smooth part of $E$ such that over every geometric point of the base scheme every irreducible component of $E$ contains a multiple of $P$ (see \cite[\S 1.4]{KM}). And a level structure $[\Gamma(2)]$ on an elliptic curve $E$ is a choice of isomorphism $\phi : \Zb/2\Zb \oplus \Zb/2\Zb \to E(2)$ where $E(2)$ is the scheme of 2-torsion Weierstrass points (i.e., kernel of the multiplication-by-2 map $[2] : E \to E$) (see \cite[II.1.18 \& IV.2.3]{DR}).

    The following Proposition shows that the fine modular curves of the above level structures are isomorphic to the weighted projective stacks $\Pc(a,b)$ under mild condition on the characteristic of the base field $K$. 



    \begin{prop}\label{prop:Moduli_t}
    The moduli stack $\Me[\Gamma]$ of generalized elliptic curves with $[\Gamma]$-level structure is isomorphic to the following over a field $K$:
        \begin{enumerate}

        \item if $\mathrm{char}(K) \neq 2$, the tame Deligne--Mumford moduli stack of generalized elliptic curves with $[\Gamma_1(2)]$-structures is isomorphic to 
        $$(\Me[\Gamma_1(2)])_K \cong [ (\Spec~K[a_2,a_4]-(0,0)) / \Gb_m ] = \Pc_K(2,4),$$  

        \item if $\mathrm{char}(K) \neq 3$, the tame Deligne--Mumford moduli stack of generalized elliptic curves with $[\Gamma_1(3)]$-structures is isomorphic to 
        $$(\Me[\Gamma_1(3)])_K \cong [ (\Spec~K[a_1,a_3]-(0,0)) / \Gb_m ] = \Pc_K(1,3),$$  

        \item if $\mathrm{char}(K) \neq 2$, the tame Deligne--Mumford moduli stack of generalized elliptic curves with $[\Gamma_1(4)]$-structures is isomorphic to 
        $$(\Me[\Gamma_1(4)])_K \cong [ (\Spec~K[a_1,a_2]-(0,0)) / \Gb_m ] = \Pc_K(1,2),$$ 

        \item if $\mathrm{char}(K) \neq 2$, the tame Deligne--Mumford moduli stack of generalized elliptic curves with $[\Gamma(2)]$-structures is isomorphic to 
        $$(\Me[\Gamma(2)])_K \cong [ (\Spec~K[a_2,a_2]-(0,0)) / \Gb_m ] = \Pc_K(2,2),$$

        \item if $\mathrm{char}(K) \nmid m$, the fine moduli space of generalized elliptic curves with $[\Gamma_1(m)]$-structures for $m=5,6,7,8,9,10 \mathrm{~or~} 12$ is isomorphic to 
        $$(\Me[\Gamma_1(m)])_K \cong [ (\Spec~K[a_1,a_1]-(0,0)) / \Gb_m ] = \Pc_K(1,1) \iso \Pb^1,$$ 

        \end{enumerate}   
        where $\lambda \cdot a_i=\lambda^i a_i$ for $\lambda \in \Gb_m$ and $i = 1,2,3,4$. Thus, the $a_i$'s have degree $i$ respectively. Moreover, the discriminant divisors of $(\Me[\Gamma])_K \cong \Pc_K(a,b)$ as above have degree 12. 
        
    \end{prop}

    \begin{proof}
    The moduli stack $\Me[\Gamma_1(2)]$ of generalized elliptic curves with $[\Gamma_1(2)]$-level structure has an isomorphism $\Me[\Gamma_1(2)] \iso \Pc(2,4)$ over $\Spec(\Zb[1/ 2])$ as in \cite[\S 1.3]{Behrens} through the universal equation 
    \[
    Y^2Z = X^3 + a_2 X^2Z + a_4 XZ^2 \,.
    \]
    And the moduli stack $\Me[\Gamma_1(3)]$ of generalized elliptic curves with $[\Gamma_1(3)]$-level structure has an isomorphism $\Me[\Gamma_1(3)] \iso \Pc(1,3)$ over $\Spec(\Zb[1/ 3])$ as in \cite[Proposition 4.5]{HMe} through the universal equation 
    \[
    Y^2Z + a_1 XYZ + a_3 YZ^2 = X^3\,.
    \]
    And the moduli stack $\Me[\Gamma_1(4)]$ of generalized elliptic curves with $[\Gamma_1(4)]$-level structure has an isomorphism $\Me[\Gamma_1(4)] \iso \Pc(1,2)$ over $\Spec(\Zb[1/ 2])$ as in \cite[Examples 2.1]{Meier} through the universal equation 
    \[
    Y^2Z + a_1 XYZ + a_1a_2 YZ^2 = X^3 + a_2X^2Z\,.
    \]
    And the moduli stack $\Me[\Gamma(2)]$ of generalized elliptic curves with $[\Gamma(2)]$-level structure has an isomorphism $\Me[\Gamma(2)] \iso \Pc(2,2)$ over $\Spec(\Zb[1/ 2])$ as in \cite[Proposition 7.1]{Stojanoska} through the universal equation (where the degree of each $\lambda_i$ is 2)
    \[
    Y^2Z = X^3 + (\lambda_1 + \lambda_2)X^2Z + \lambda_1\lambda_2XZ^2\,.
    \]
    Finally, the fine moduli space $\Me[\Gamma(m)]$ of generalized elliptic curves with $[\Gamma(m)]$-level structure for $m=5,6,7,8,9,10 \mathrm{~or~} 12$ has an isomorphism $\Me[\Gamma(m)] \iso \Pb^1$ over $\Spec(\Zb[1/ m])$ as in \cite[Example 2.5]{Meier}.

    By Remark~\ref{rmk:wtprojtame}, the respective $\Me[\Gamma]$ as weighted projective stacks are tame Deligne--Mumford as well, and in fact, smooth.
    
    For the degree of the discriminant, it suffices to find the weight of the $\Gb_m$-action. First, the four papers cited above explicitly construct universal families of elliptic curves over the schematic covers $(\mathrm{Spec}\; K[a_i,a_j] - (0,0)) \rightarrow \Pc_K(i,j)$ of the corresponding moduli stacks. The explicit defining equation of the respective universal family implies that the $\lambda \in \Gb_m$ also acts on the discriminant of the universal family by multiplying $\lambda^{12}$. Therefore, the discriminant has degree 12.
    \end{proof}

    A generalized elliptic curve $X$ over $C_K$ can be thought of as a flat family of semistable elliptic curves admitting a group structure, such that a finite group scheme $\mathcal{G} \rightarrow C_K$ (determined by $\Gamma$) embeds into $X$ and its image meets every irreducible component of every geometric fibers of $X$. 

    \medskip

    We now consider the moduli stack $\Lc_{12n,g}^{[\Gamma]} := \Hom_n(C,\Me[\Gamma])$ of generalized elliptic fibrations over $C$ with $[\Gamma]$-level structures and $12n$ nodal singular fibers. 

    \begin{prop}\label{Lef_Moduli_Space_t}
        Let $n \in \mathbb{Z}_{+}$ and $K$ be a field with $\mathrm{char}(K) = 0$ or $\mathrm{char}(K) \nmid m$ where $m$ depends on $[\Gamma]$ as Proposition~\ref{prop:Moduli_t}. Then the moduli stack $\Lc_{12n,g}^{[\Gamma]}$ of generalized elliptic fibrations over the parameterized smooth projective basecurve $C_{K}$ of genus $g$ with discriminant degree $12n>0$ and $[\Gamma]$-level structures is the tame Deligne--Mumford stack $\Hom_{n}(C,\Me[\Gamma])$ parameterizing the $K$-morphisms $f:C \rightarrow \Me[\Gamma]$ such that $f^*\Oc_{\Me[\Gamma]}(1) \in \Pic^n C$ . 
    \end{prop} 

    \begin{proof}  
    Without the loss of generality, we prove the $\Hom_n(C,\Me[\Gamma_1(2)])$ case over a field $K$ with $\mathrm{char}(K) \neq 2$. The proof for the other cases are analogous. By the definition of the universal family $p$, any generalized elliptic curves $\pi: Y \rightarrow C$ with $[\Gamma_1(2)]$-structures comes from a morphism $f:C \rightarrow \Me[\Gamma_1(2)]$ and vice versa. As this correspondence also works in families, the moduli stack of generalized elliptic curves over $C$ with $[\Gamma_1(2)]$-structures is isomorphic to $\mathrm{Hom}(C,\Me[\Gamma_1(2)])$.
    
    Since the discriminant degree of $f$ is $12\deg f^*\Oc_{\Me[\Gamma_1(2)]}(1)$ by Proposition~\ref{prop:Moduli_t}, the substack $\mathrm{Hom}_n(C, \Me[\Gamma_1(2)])$ parametrizing such $f$'s with $\deg f^*\Oc_{\Me[\Gamma_1(2)]}(1)=n$ is the desired moduli stack. Since $\deg f^*\Oc_{\Me[\Gamma_1(2)]}(1)=n$ is an open condition, $\mathrm{Hom}_n(C, \Me[\Gamma_1(2)])$ is an open substack of $\mathrm{Hom}(C, \Me[\Gamma_1(2)])$, which is tame Deligne--Mumford as $\Me[\Gamma_1(2)]$ itself is tame Deligne--Mumford by Proposition~\ref{prop:Moduli_t}. Thus $\mathrm{Hom}_n(C, \Me[\Gamma_1(2)])$ satisfies the desired properties.
    \end{proof}

    We now acquire the exact number $|\Lc_{12n,0}^{[\Gamma]}(\Fb_q)/\sim|$ of $\Fb_q$--isomorphism classes of $\Fb_q$--points (i.e., the non--weighted point count) of the moduli stack $\Lc_{12n,0}^{[\Gamma]}$ .


    \begin{thm} \label{thm:non_weighted_point_torsion}
        If $\mathrm{char}(\Fb_q) \neq 2$, then \\ 
        \begingroup
        \allowdisplaybreaks
        \begin{align*}
            |\Lc_{12n,0}^{[\Gamma_1(2)]}(\Fb_q)/\sim|&= 2 \cdot \#_q\left(\Hom_n(\Pb^1,\Pc(2,4))\right) = 2 (q^{6n+1} - q^{6n-1})\\\\
            |\Lc_{12n,0}^{[\Gamma_1(4)]}(\Fb_q)/\sim|&= \#_q\left(\Hom_n(\Pb^1,\Pc(1,2))\right) = q^{3n+1} - q^{3n-1}\\\\
            |\Lc_{12n,0}^{[\Gamma(2)]}(\Fb_q)/\sim|&= 2 \cdot \#_q\left(\Hom_n(\Pb^1,\Pc(2,2))\right) = 2(q^{4n+1} - q^{4n-1})
        \end{align*}
        \endgroup
        If $\mathrm{char}(\Fb_q) \neq 3$, then
        \[|\Lc_{12n,0}^{[\Gamma_1(3)]}(\Fb_q)/\sim|= \#_q\left(\Hom_n(\Pb^1,\Pc(1,3))\right) = q^{4n+1} - q^{4n-1} \]
        If $\mathrm{char}(\Fb_q) \nmid m$, then
        \[|\Lc_{12n,0}^{[\Gamma_1(m)]}(\Fb_q)/\sim|= \#_q\left(\Hom_n(\Pb^1,\Pc(1,1) \iso \Pb^1)\right) = q^{2n+1} - q^{2n-1} \]
        $m$ is for $m=5,6,7,8,9,10 \mathrm{~or~} 12$.
    \end{thm}

    \begin{proof}\label{pf:gen_stab}
    Fix $n \in \mathbb{Z}_{\geq 1}$. Since any $\varphi_g \in \mathrm{Hom}_n(\Pb^1,\Pc(a,b))$ is surjective, the generic stabilizer group $\mu_{\gcd(a,b)}$ of $\Pc(a,b)$ is the automorphism group of $\varphi_g$. Using the identification from Proposition~\ref{Lef_Moduli_Space_t} and the weighted point counts (see Definition \ref{def:wtcount}) of $\Hom$ stacks as in Theorem \ref{P-cohomThm}, we have the desired formula for the $\Fb_q$--isomorphism classes of $\Fb_q$--points (i.e., the non--weighted point count) as  
        \[|\Lc_{12n,0}^{[\Gamma]}(\Fb_q)/\sim|=|\mu_{\gcd(a,b)}| \cdot (q^{(a+b)n+1}-q^{(a+b)n-1})\;\] 
        where the factor 2 comes from the hyperelliptic involution when $\mu_{\gcd(a,b)} = \mu_2$ .
    \end{proof}

    
    Again, we only consider the \textit{non-isotrivial} generalized elliptic fibrations. 

    The $\Delta$ is the discriminant of a generalized elliptic fibration and if $K=\Fb_q$, then $0<ht(\Delta):=q^{\deg \Delta} = q^{12n}$.
    Now, define $\Nc(\Fb_q(t),~[\Gamma],0 < q^{12n} \le B)$ 
     \[\coloneqq |\{\text{Generalized elliptic curves over } \Pb^{1}_{\Fb_q} \text{ with } [\Gamma] \text{--structures and } 0<ht(\Delta) \le B\}|\]
    Then, we acquire the following descriptions of $\Nc(\Fb_q(t),~[\Gamma],0 < q^{12n} \le B)$:

    \begin{cor}\label{cor:ell_curve_count_Gamma}
    The function $\Nc(\Fb_q(t),~[\Gamma],0 < B \le q^{12n})$, which counts the number of generalized elliptic curves with $[\Gamma]$-level structures over $\Pb^1_{\Fb_q}$ with $\mathrm{char}(\Fb_q) \nmid m$ where $m$ depends on $\Gamma$ as Proposition~\ref{prop:Moduli_t}; ordered by $0< ht(\Delta) = q^{12n} \le B$, satisfies:\\
        \begingroup
        \allowdisplaybreaks
        \begin{align*}
        \Nc(\Fb_q(t),~[\Gamma_1(2)],~0 < q^{12n} \le B) & = 2 \cdot \frac{(q^{7} - q^{5})}{(q^{6}-1)} \cdot \left( B^{\frac{1}{2}} - 1 \right)\\\\
        \Nc(\Fb_q(t),~[\Gamma_1(3)],~0 < q^{12n} \le B) & = \frac{(q^{5} - q^{3})}{(q^{4}-1)} \cdot \left( B^{\frac{1}{3}} - 1 \right)\\\\
        \Nc(\Fb_q(t),~[\Gamma_1(4)],~0 < q^{12n} \le B) & = \frac{(q^{4} - q^{2})}{(q^{3}-1)} \cdot \left( B^{\frac{1}{4}} - 1 \right)\\\\
        \Nc(\Fb_q(t),~[\Gamma(2)],~0 < q^{12n} \le B) & = 2 \cdot \frac{(q^{5} - q^{3})}{(q^{4}-1)} \cdot \left( B^{\frac{1}{3}} - 1 \right)\\\\
        \Nc(\Fb_q(t),~[\Gamma_1(m)],~0 < q^{12n} \le B) & = \frac{(q^{3} - q^{1})}{(q^{2}-1)} \cdot \left( B^{\frac{1}{6}} - 1 \right)
        \end{align*}
        \endgroup
        $m$ is for $m=5,6,7,8,9,10 \mathrm{~or~} 12$.
    \end{cor} 

    \begin{proof}
        Without the loss of the generality, we prove the $[\Gamma_1(2)]$--level structure case over $\mathrm{char}(\Fb_q) \neq 2$. The proof for the other cases are analogous. By Theorem~\ref{thm:non_weighted_point_torsion}, we know that the number of $\Fb_q$-isomorphism classes of generalized elliptic fibrations of discriminant degree $12n$ with $[\Gamma_1(2)]$-structures over $\Pb^1_{\Fb_q}$ is $|\Lc_{12n,0}^{[\Gamma_1(2)]}(\Fb_q)/\sim| = 2 \cdot (q^{6n + 1} - q^{6n - 1})$. Using this, we can explicitly compute the sharp enumeration on $\Nc(\Fb_q(t),~[\Gamma_1(2)],~0 < q^{12n} \le B)$ as follows
        \begingroup
        \allowdisplaybreaks
        \begin{align*}
        \Nc(\Fb_q(t),~[\Gamma_1(2)],~0 < q^{12n} \le B) & = \sum \limits_{n=1}^{\left \lfloor \frac{log_q B}{12} \right \rfloor} |\Lc_{12n,0}^{[\Gamma_1(2)]}(\Fb_q)/\sim| = 2 \cdot \frac{(q^{7} - q^{5})}{(q^{6}-1)} \cdot ( B^{\frac{1}{2}} - 1)
        \end{align*}
        \endgroup
    \end{proof}

    \par The main leading term of the sharp enumerations over $\Fb_q(t)$ matches the analogous asymptotic counts ordered by na\"ive height of underlying elliptic curves over $\Q$ by Harron and Snowden in \cite[Theorem 1.2]{HS} (see also \cite{Duke, Grant}). Results over higher genus $C_{\Fb_q}$ where we acquire the corresponding stable \'etale cohomology followed by stable point counts as in Corollary \ref{Ell_Count} is straightforward.

    \medskip

    \subsection*{Arithmetic moduli of $(m-1)$-stable genus one fibrations over \texorpdfstring{$C$}{C} with prescribed \texorpdfstring{$m$}{m}-marked points}\label{subsec:ptcount_m}

    \par We proceed to determine the sharp enumeration on the number of $m$-marked $(m-1)$-stable genus one fibrations over $\Pb^1_{\Fb_q}$ for $2 \le m \le 5$. First, we state the definition of $m$-marked $(m-1)$-stability from \cite[Definition 1.5.3]{LP}, which is a modification of the Deligne--Mumford stability \cite{DM}:

    \begin{defn}\label{def:m-pted_l-stable_curve}
        Let $K$ be a field and $m$ be a positive integer. Then, a tuple $(C,p_1,\dotsc,p_m)$, of a geometrically connected, geometrically reduced, and proper $K$-curve $C$ of arithmetic genus one with $m$ distinct $K$-rational points $p_i$ in the smooth locus of $C$, is a $(m-1)$-stable $m$-marked curve of arithmetic genus one if the curve $C_{\overline{K}} :=C \times_K \overline{K}$ and the divisor $\Sigma:=\{p_1,\dotsc,p_m\}$ satisfy the following properties, where $\overline{K}$ is the algebraic closure of $K$:
        \begin{enumerate}
            \item $C_{\overline{K}}$ has only nodes and elliptic $u$-fold points as singularities (see below), where $u < m$,
            \item $C_{\overline{K}}$ has no disconnecting nodes, and
            \item every irreducible component of $C_{\overline{K}}$ contains at least one marked point.
        \end{enumerate}
    \end{defn}
    
    \begin{rmk}\label{rmk:m-pted_l-stable_curve}
        A singular point of a curve over $\overline{K}$ is an elliptic $u$-fold singular point if it is Gorenstein and \'etale locally isomorphic to a union of $u$ general lines in $\Pb^{u-1}_{\overline{K}}$ passing through a common point.
    \end{rmk}
    
    \par Note that the name ``$(m-1)$-stability'' comes from \cite[\S 1.1]{Smyth}, which is defined when $\mathrm{char}(K) \neq 2,3$. By \cite[Proposition 1.5.4]{LP}, the above definition (by \cite[Definition 1.5.3]{LP}) coincides with that of Smyth when $\mathrm{char}(K) \neq 2,3$, hence we adapt Smyth's naming convention on Lekili and Polishchuk's definition. Regardless, we focus on the case when $\mathrm{char}(K) \neq 2,3$, so that the moduli stack of such curves behaves reasonably.
    
    \medskip
    
    \par By \cite[Theorem 3.8]{Smyth}, we have the moduli stack of $(m-1)$-stable $m$-marked curves of arithmetic genus one over any field of characteristic $\neq 2,3$:

    \begin{thm}\label{def:Smyth}
        There exists a proper irreducible Deligne--Mumford moduli stack $\Mg_{1,m}(m-1)$ of $(m-1)$-stable $m$-marked curves arithmetic genus one over $\Spec(\Zb[1/6])$
    \end{thm}
    
    \par Note that when $m=1$, $\Mg_{1,1}(0) \cong \Mg_{1,1}$ is the Deligne--Mumford moduli stack of stable elliptic curves.
    
    \par In fact, the construction of $\Mg_{1,m}(m-1)$ extends to $\Spec~\Z$ by \cite[Theorem 1.5.7]{LP} (called $\Mg_{1,m}^{\infty}$ in loc.cit.) as an algebraic stack, which is proper over $\Spec~\Z[1/N]$ where $N$ depends on $m$:
    \begin{itemize}
        \item if $m \ge 3$, then $N=1$,
        \item if $m=2$, then $N=2$, and
        \item if $m=1$, then $N=6$.
    \end{itemize}
    However, even with those assumptions above, $\Mg_{1,m}(m-1)$ is not necessarily Deligne--Mumford. Nevertheless, by \cite[Theorem 1.5.7.]{LP}, we obtain the explicit descriptions of $\Mg_{1,m}(m-1)$:

    \begin{prop} \label{prop:Moduli_s}
    The moduli stack $\Mg_{1,m}(m-1)$ of $m$-marked $(m-1)$-stable curves of arithmetic genus one for $2 \le m \le 5$ is isomorphic to the following over a field $K$:
        \begin{enumerate}

        \item if $\mathrm{char}(K) \neq 2,3$, the tame Deligne--Mumford moduli stack of 2-marked 1-stable curves of arithmetic genus one is isomorphic to $$(\Mg_{1,2}(1))_K \cong [ (\Spec~K[a_2,a_3,a_4] - 0) / \Gb_m ] = \Pc_K(2,3,4),$$  

        \item if $\mathrm{char}(K) \neq 2,3$, the tame Deligne--Mumford moduli stack of 3-marked 2-stable curves of arithmetic genus one is isomorphic to $$(\Mg_{1,3}(2))_K \cong [ (\Spec~K[a_1,a_2,a_2,a_3] - 0) / \Gb_m ] = \Pc_K(1,2,2,3),$$  

        \item if $\mathrm{char}(K) \neq 2$, the tame Deligne--Mumford moduli stack of 4-marked 3-stable curves of arithmetic genus one is isomorphic to $$(\Mg_{1,4}(3))_K \cong [ (\Spec~K[a_1,a_1,a_1,a_2,a_2] - 0) / \Gb_m ] = \Pc_K(1,1,1,2,2),$$  

        \item over any $K$, the fine moduli space of 5-marked 4-stable curves of arithmetic genus one is isomorphic to a scheme $$(\Mg_{1,5}(4))_K \cong [ (\Spec~K[a_1,a_1,a_1,a_1,a_1,a_1] - 0) / \Gb_m ] = \Pb_K(1,1,1,1,1,1) \iso \Pb_K^5, $$  
        \end{enumerate}   
        where $\lambda \cdot a_i=\lambda^i a_i$ for $\lambda \in \Gb_m$ and $i = 1,2,3,4$. Thus, the $a_i$'s have degree $i$ respectively. Furthermore, if $\mathrm{char}(K) \neq 2,3$, then the discriminant divisors of such $\Mg_{1,m}(m-1)$ have degree 12.
    \end{prop}

    \begin{proof}
    The moduli stack $\Mg_{1,2}(1)$ of 2-marked points at $\infty$ and $(0,0)$ Smyth's 1-stable curves of arithmetic genus one has an isomorphism $\Mg_{1,2}(1) \iso \Pc(2,3,4)$ over $\Spec(\Zb[1/ 6])$ as in \cite[Theorem 1.5.7.]{LP} through the universal equation 
    \[
    Y^2Z + a_3 YZ^2 = X^3 + a_2 X^2Z + a_4 XZ^2 \,,
    \]
    with discriminant $\Delta = -16a_2^3a_3^2 + 16a_2^2a_4^2 - 64a_4^3 - 27a_3^4 + 56a_2a_4a_3^2$. Similarly, the Proof of \cite[Theorem 1.5.7.]{LP} gives the corresponding isomorphisms $\Mg_{1,m}(m-1) \iso \Pcv$. 

    By Remark~\ref{rmk:wtprojtame}, the respective $\Me[\Gamma]$ as weighted projective stacks are tame Deligne--Mumford as well, and in fact, smooth.
    
    For the degree of the discriminant when $\mathrm{char}(K) \neq 2,3$, it suffices to describe the discriminant divisor, the locus of singular curves in $\Mg_{1,m}(m-1)$. First, \cite[Theorem 1.5.7.]{LP} shows that in the above case, where $\Mg_{1,m}(m-1) \cong \Pcv$, the line bundle $\Oc_{\Pcv}$(1) of degree one is isomorphic to $\lambda:=\pi_*\omega_{\pi}$, where $\pi: \overline{\Cc}_{1,m}(m-1) \rightarrow \Mg_{1,m}(m-1)$ is the universal family of $(m-1)$-stable $m$-marked curves of arithmetic genus one. Since $\Mg_{1,m}(m-1)$ is smooth and the Picard rank is one (generated by $\lambda$), the discriminant divisor is Cartier. In fact, by \cite[\S 3.1]{Smyth2}, it coincides with the locus $\Delta_{irr}$ of curves with non-disconnecting nodes or non-nodal singular points. Then \cite[Remark 3.3]{Smyth2} (which assumes $\mathrm{char}(K) \neq 2,3$) implies that $\Delta_{irr} \sim 12\lambda$, thus the discriminant divisor has degree 12.
    \end{proof}

    We now consider the moduli stack $\Lc^{m}_{12n,0} \coloneqq \mathrm{Hom}_n(C, \Mg_{1,m}(m-1))$ of $m$-marked $(m-1)$-stable genus one fibrations over $C$ with $12n$ nodal singular fibers. 

    \begin{prop}\label{Lef_Moduli_Space2}
    Let $n \in \mathbb{Z}_{+}$ and $K$ be a field with $\mathrm{char}(K) = 0$ or $\mathrm{char}(K) > 3$. Then the moduli stack $\Lc^{m}_{12n,0}$ of $m$-marked $(m-1)$-stable genus one fibrations over the parameterized smooth projective basecurve $C_{K}$ of genus $g$ with discriminant degree $12n>0$ is the tame Deligne--Mumford stack $\mathrm{Hom}_n(C, \Mg_{1,m}(m-1))$ parameterizing the $K$-morphisms $f:C \rightarrow \Mg_{1,m}(m-1)$ such that $f^*\Oc_{\Mg_{1,m}(m-1)}(1) \in \Pic^n C$ . 
    \end{prop} 

    \begin{proof}
    Without the loss of the generality, we prove the $2$-marked $1$-stable case over a field $K$ with $\mathrm{char}(K) \neq 2,3$. The proof for the other cases are analogous. By the definition of the universal family $p$, any $2$-marked $1$-stable arithmetic genus one curves $\pi: Y \rightarrow C$ with discriminant degree $12n$ comes from a morphism $f:C \rightarrow \Mg_{1,2}(1)$ and vice versa. As this correspondence also works in families, the moduli stack of $2$-marked $1$-stable curves of arithmetic genus one over $C_K$ is isomorphic to $\mathrm{Hom}(C,\Mg_{1,2}(1))$. 

    Since the discriminant degree of $f$ is $12\deg f^*\Oc_{\Mg_{1,2}(1)}(1)$ by Proposition~\ref{prop:Moduli_s}, the substack $\mathrm{Hom}_n(C,\Mg_{1,2}(1))$ parametrizing such $f$'s with $\deg f^*\Oc_{\Mg_{1,2}(1)}(1)=n$ is the desired moduli stack. Since $\deg f^*\Oc_{\Mg_{1,2}(1)}(1)=n$ is an open condition, $\mathrm{Hom}_n(C, \Mg_{1,2}(1))$ is an open substack of $\mathrm{Hom}(C, \Mg_{1,2}(1))$, which is tame Deligne--Mumford by \cite[Proposition 3.6]{HP2} as $\Mg_{1,2}(1)$ itself is tame Deligne--Mumford by Proposition~\ref{prop:Moduli_s}. Thus $\mathrm{Hom}_n(C, \Mg_{1,2}(1))$ satisfies the desired properties.
    \end{proof}

    We now acquire the exact number $|\Lc^{m}_{12n,0}(\Fb_q)/\sim|$ of $\Fb_q$--isomorphism classes of $\Fb_q$--points (i.e., the non--weighted point count) of the moduli stack $\Lc^{m}_{12n,0}$ .

    \begin{thm} \label{thm:non_weighted_point_m_sections}
        If $\mathrm{char}(\Fb_q) \neq 2,3$, then \\ 
        \begingroup
        \allowdisplaybreaks
        \begin{align*}
            |\Lc^{m=2}_{12n,0}(\Fb_q)/\sim| &= \#_q\left(\Hom_n(\Pb^1,\Pc(2,3,4))\right) + \#_q\left(\Hom_n(\Pb^1,\Pc(2,4))\right)\\
            \phantom{=}\:& = (q^{9n + 2} + q^{9n + 1} - q^{9n - 1} - q^{9n - 2}) + (q^{6n + 1} - q^{6n - 1})\\\\
            |\Lc^{m=3}_{12n,0}(\Fb_q)/\sim| &= \#_q\left(\Hom_n(\Pb^1,\Pc(1,2,2,3))\right) + \#_q\left(\Hom_n(\Pb^1,\Pc(2,2))\right)\\
            \phantom{=}\:& = (q^{8n+3}+q^{8n+2}+q^{8n+1}-q^{8n-1}-q^{8n-2}-q^{8n-3}) + (q^{4n + 1} - q^{4n - 1})\\\\
            |\Lc^{m=4}_{12n,0}(\Fb_q)/\sim| &= \#_q\left(\Hom_n(\Pb^1,\Pc(1,1,1,2,2))\right) + \#_q\left(\Hom_n(\Pb^1,\Pc(2,2))\right)\\
            \phantom{=}\:& = (q^{7n+4}+q^{7n+3}+q^{7n+2}+q^{7n+1}-q^{7n-1}-q^{7n-2}-q^{7n-3}-q^{7n-4}) \\
            &\phantom{= (} + (q^{4n + 1} - q^{4n - 1})\\
            |\Lc^{m=5}_{12n,0}(\Fb_q)/\sim| &= \#_q\left(\Hom_n(\Pb^1,\Pb(1,1,1,1,1,1) \iso \Pb^5)\right) \\
            \phantom{=}\:& = q^{6n+5}+q^{6n+4}+q^{6n+3}+q^{6n+2}+q^{6n+1} -q^{6n-1}-q^{6n-2}-q^{6n-3}\\&\,\,\,\,\,\,\,\,\,\, -q^{6n-4}-q^{6n-5}
        \end{align*}
        \endgroup
    \end{thm}

    \begin{proof} 
        Note that $\Mg_{1,2}(1) \cong \Pc(2,3,4)$ has the substack $\Pc(2,4)$ with the generic stabilizer of order 2. Using the identification from Proposition~\ref{Lef_Moduli_Space2} and the weighted point counts (see Definition \ref{def:wtcount}) of $\Hom$ stacks as in Theorem \ref{P-cohomThm}, we have the number of isomorphism classes of $\Fb_q$-points of $\Lc^{m=2}_{12n,0}$ with discriminant degree $12n$ is $|\Lc^{m=2}_{12n,0}(\Fb_q)/\sim| = (q^{9n + 2} + q^{9n + 1} - q^{9n - 1} - q^{9n - 2}) + (q^{6n + 1} - q^{6n - 1})$ by summing the weighted point counts of $\Hom$ stacks as in \cite[Proposition 4.10]{HP2}. Similarly, $\Mg_{1,3}(2) \cong \Pc(1,2,2,3)$ and $\Mg_{1,4}(3) \cong \Pc(1,1,1,2,2)$ has the substack $\Pc(2,2)$ with the generic stabilizer of order 2. This implies that adding $(q^{4n + 1} - q^{4n - 1})$ to the corresponding weighted points count gives the desired non--weighted point counts.  Finally, $\Mg_{1,5}(4) \cong \Pb^5$, so that the non-weighted point count coincides with the weighted point count.
    \end{proof}

    \medskip

    Again, we only consider the \textit{non-isotrivial} $m$-marked $(m-1)$-stable genus one curves. The $\Delta$ is the discriminant of a $m$-marked $(m-1)$-stable genus one and if $K=\Fb_q$, then $0<ht(\Delta):=q^{\deg \Delta} = q^{12n}$.
    Now, define $\Nc(\Fb_q(t),~m,~0 < B \le q^{12n})$ 
     \[\coloneqq |\{m\text{-marked } (m-1) \text{-stable genus one fibrations over } \Pb^{1}_{\Fb_q} \text{ with } 0<ht(\Delta) \le B\}|\]
     Note that when $m=1$, $\Nc(\Fb_q(t),~m=1,~0 < B \le q^{12n})$ counts the stable elliptic fibrations as in \cite[Theorem 3]{HP}. When $2 \le m \le 5$, we acquire the following sharp enumerations of $\Nc(\Fb_q(t),~m,~0 < B \le q^{12n})$ abbreviated as $\Nc(\Fb_q(t),~m,~B)$:

    \begin{cor}\label{cor:ell_curve_m_count}
    The function $\Nc(\Fb_q(t),~m,~0 < B \le q^{12n})$, which counts the number of $m$-marked $(m-1)$-stable genus one fibration over $\Pb^1_{\Fb_q}$ with $\text{char} (\Fb_q) \neq 2,3$ ordered by $0< ht(\Delta) = q^{12n} \le B$, satisfies:

        \begingroup
        \allowdisplaybreaks
        \begin{align*}
        \Nc(\Fb_q(t),~m=2,~B) & = \frac{({q^{11} + q^{10} - q^{8} - q^{7}})}{(q^{9}-1)} \cdot ( B^{\frac{3}{4}} - 1) +  \frac{(q^{7} - q^{5})}{(q^{6}-1)} \cdot ( B^{\frac{1}{2}} - 1) \\\\
        \Nc(\Fb_q(t),~m=3,~B) & = \frac{({q^{11} + q^{10} + q^{9} - q^{7} - q^{6} - q^{5}})}{(q^{8}-1)} \cdot ( B^{\frac{2}{3}} - 1) + \frac{(q^{5} - q^{3})}{(q^{4}-1)} \cdot ( B^{\frac{1}{3}} - 1) \\\\
        \Nc(\Fb_q(t),~m=4,~B) & =\frac{({q^{11} + q^{10} + q^{9} + q^{8} - q^{6} - q^{5} - q^{4} - q^{3}})}{(q^{7}-1)} \cdot ( B^{\frac{7}{12}} - 1)  \\ &\,\,\,\,\,\, +\frac{(q^{5} - q^{3})}{(q^{4}-1)} \cdot ( B^{\frac{1}{3}} - 1) \\\\
        \Nc(\Fb_q(t),~m=5,~B) & = \frac{({q^{11} + q^{10} + q^{9} + q^{8} + q^{7} -q^{5}- q^{4} - q^{3} - q^{2} - q^{1}})}{(q^{6}-1)} \cdot ( B^{\frac{1}{2}} - 1)
        \end{align*}
        \endgroup
    \end{cor} 

    \begin{proof}
        Without the loss of the generality, we prove the $2$-marked $1$-stable case over $\mathrm{char}(\Fb_q) \neq 2,3$. The proof for the other cases are analogous. By Theorem~\ref{thm:non_weighted_point_m_sections}, we know that the number of $\Fb_q$-isomorphism classes of $1$-stable arithmetic genus one curves over $\Pb^1_{\Fb_q}$ with $2$-marked sections is $|\Lc^{m=2}_{12n,0}(\Fb_q)/\sim| = (q^{9n + 2} + q^{9n + 1} - q^{9n - 1} - q^{9n - 2}) + (q^{6n + 1} - q^{6n - 1})$. Using this, we can explicitly compute the sharp enumeration on $\Nc(\Fb_q(t),~m=2,~0 < q^{12n} \le B)$ as follows

        \begingroup
        \allowdisplaybreaks
        \begin{align*}
        \Nc(\Fb_q(t),~m=2,~0 < q^{12n} \le B) & = \sum \limits_{n=1}^{\left \lfloor \frac{log_q B}{12} \right \rfloor} |\Lc^{m=2}_{12n,0}(\Fb_q)/\sim| \\
        & = \frac{({q^{11} + q^{10} - q^{8} - q^{7}})}{(q^{9}-1)} \cdot ( B^{\frac{3}{4}} - 1) + \frac{(q^{7} - q^{5})}{(q^{6}-1)} \cdot ( B^{\frac{1}{2}} - 1)
        \end{align*}
        \endgroup
    \end{proof}

    Results over higher genus $C_{\Fb_q}$ where we acquire the corresponding stable \'etale cohomology followed by stable point counts as in Corollary \ref{Ell_Count} is straightforward.

    \bigskip


    \section*{Acknowledgements}
    The authors are indebted to Peter Scholze for helpful pointers and discussions, especially for his help with transferring proper descent to the world of Deligne-Mumford stacks, some of which made its way into the manuscript. Warm thanks to Dori Bejleri, Changho Han, David Hansen, Jesse Wolfson and Craig Westerland as well for earlier helpful discussions and to Benson Farb and Burt Totaro for suggesting edits that made the paper more readable. Oishee Banerjee is supported by Hausdorff Center of Mathematics, Bonn. Jun-Yong Park was supported by the Institute for Basic Science in Korea (IBS-R003-D1) and the Max Planck Institute for Mathematics. Johannes Schmitt was supported by the grant 184613 of the Swiss National Science Foundation.


    
    \vspace{+16 pt}

    \noindent Oishee Banerjee \\
    \textsc{Mathematisches Institut, Universit\"at Bonn, Endenicher Allee 60, \\ 53115 Bonn, Germany} \\
      \textit{E-mail address}: \texttt{oishee@math.uni-bonn.de}

    \vspace{+16 pt}

    \noindent Jun--Yong Park \\
    \textsc{Max-Planck-Institut f\"ur Mathematik, Vivatsgasse 7, \\ 53111 Bonn, Germany} \\
      \textit{E-mail address}: \texttt{junepark@mpim-bonn.mpg.de}

    \vspace{+16 pt}

    \noindent Johannes Schmitt \\
    \textsc{Institut f\"ur Mathematik, University of Z\"urich, Winterthurerstrasse 190, 8057 Z\"urich, Switzerland} \\
      \textit{E-mail address}: \texttt{johannes.schmitt@math.uzh.ch}

\end{document}